\documentclass[reqno,12pt]{amsart}
\usepackage{latexsym}
\usepackage{amsfonts}
\usepackage{amssymb}
\usepackage{mathrsfs}
\usepackage[all]{xy}
\usepackage{bbm}
\usepackage{enumitem}
\usepackage{color}
\usepackage{mathtools}
\usepackage[mathscr]{euscript}
\usepackage{tikz}
\usepackage{amsmath,amsthm,amssymb,mathrsfs,amsfonts,verbatim,color,leftidx}
\usepackage{amsbsy}
\usepackage[colorlinks,urlcolor=blue,linkcolor=blue,citecolor=brown]{hyperref}
\usetikzlibrary{matrix,arrows}
\usepackage{color}
\definecolor{jasper}{rgb}{0,0.5,0}

\newtheorem{theorem}{Theorem}[section]
\newtheorem{corollary}[theorem]{Corollary}
\newtheorem{lemma}[theorem]{Lemma}
\newtheorem{conjecture}[theorem]{Conjecture}

\newtheorem{question}[theorem]{Question}
\newtheorem{proposition}[theorem]{Proposition}
\theoremstyle{definition}

\newtheorem{remark}[theorem]{Remark}
\numberwithin{equation}{section}

\newcommand{\RR}{\mathbb{R}}

\newcommand{\NN}{\mathbb{N}}

\newcommand{\QQ}{\mathbb{Q}}
\newcommand{\ZZ}{\mathbb{Z}}

\newcommand{\cR}{\mathcal{R}}

\newcommand{\cL}{\mathcal{L}}

\newcommand{\cN}{\mathcal{N}}

\newcommand{\cW}{\mathcal{W}}

\newcommand{\bu}{\mathbf{u}}
\newcommand{\bv}{\mathbf{v}}
\newcommand{\bw}{\mathbf{w}}

\newcommand{\diag}{\mathrm{diag}}

\newcommand{\dist}{\mathrm{dist}}

\newcommand{\SL}{\mathrm{SL}}

\newcommand{\Ad}{\mathrm{Ad}}
\newcommand{\ad}{\mathrm{ad}}

\newcommand{\Lie}{\mathrm{Lie}}

\newcommand{\rk}{\mathrm{rank}}
\newcommand{\Span}{\mathrm{Span}}

\renewcommand{\ker}{\mathrm{Ker}}

\newcommand{\E}{\mathscr{E}}

\newcommand{\Lg}{\mathfrak{g}}
\newcommand{\Lh}{\mathfrak{h}}

\newcommand{\la}{\langle}
\newcommand{\ra}{\rangle}

\newcommand{\eps}{\varepsilon}

\newcommand {\ignore}[1]  {}
\newcommand {\new}[1]   {\textcolor{blue}{#1}}
\newcommand {\comm}[1]   {\textcolor{red}{#1}}
\newcommand {\neww}[1]   {\textcolor{purple}{#1}}

\newif\ifdraft\drafttrue
\draftfalse

\newcommand\eq[2]{{\ifdraft{\ \tt [#1]}\else\ignorespaces\fi}\begin{equation}\label{eq:#1}{#2}\end{equation}}

\newcommand {\equ}[1]     {\eqref{eq:#1}}
\newcommand{\ggm}{G/\Gamma}

\newcommand\eqx[2]{{\ifdraft{\ \tt [#1]}\else\ignorespaces\fi}\begin{equation*}\label{eq:#1}{#2}\end{equation*}}

\renewcommand{\setminus}{\smallsetminus}

\begin{document}

\title{Simultaneously bounded and dense orbits for commuting Cartan actions}
\author{Dmitry Kleinbock}
\address{Department of Mathematics, Brandeis University, Waltham MA}
\email{kleinboc@brandeis.edu}

\author{Chengyang Wu}
\address{School of Mathematical Sciences, Peking University, Beijing, 100871, China}
\email{chengyangwu1999@gmail.com}
\date{August 24, 2025}
\thanks{The first-named author was supported by NSF grant DMS-2155111.}

\begin{abstract}
    In this paper we prove that the set of points that have bounded orbits under one regular diagonal flow and dense orbits under the other diagonal flow commuting with the first one has full Hausdorff dimension in $X_3=\SL_3(\RR)/\SL_3(\ZZ)$. 

    To explain its application towards the Uniform Littlewood's Conjecture proposed in \cite{BFK25}, we introduce the concept of ``fiberwise nondivergence'' for the action of a cone inside the full diagonal subgroup. Then our main result implies that there exists a dense subset of $X_3$ in which each point has a fiberwise non-divergent orbit under a cone inside the full diagonal subgroup and an unbounded orbit under every diagonal flow. 
\end{abstract}

\maketitle

\section{Introduction}

\subsection{Simultaneously dense and nondense orbits}

%\comm{A draft of the introduction:} 
Let $X$ be a metric space, and let $F$ be a one-parameter group or semigroup of self-maps of $X$.  We will denote by $D(F)$ the set of points with a dense $F$-orbit, and by $ND(F)$ its complement, i.e.\ the set of points with a nondense $F$-orbit. A natural question one could ask  is how large the sets $D(F)$   and   $ND(F)$ can be. If the action admits an ergodic invariant probability measure $\mu$, or has some hyperbolic behavior, some instances of this question can be  answered. Namely, under the assumption of ergodicity the set  $D(F)$ has full measure, and for many hyperbolic systems one can prove %the property 
that $ND(F)$ is winning in the sense of Schmidt games (see \cite{BFK11,Da88,T09,T16,wu16,wu18}).
From that it follows that both $D(F)$   and   $ND(F)$ are {\sl thick}, that is, their intersection with any non-empty open set has full Hausdorff dimension. 
Moreover, both full-measure and winning conditions are stable with respect to countable intersections. In particular it implies that for any choice of countably many semigroups $F_i$ for which the above conclusions can be established, it holds that both $\bigcap_i D(F_i)$   and   $\bigcap_i ND(F_i)$  are thick.

\smallskip
However one can also consider 
a mixed case, that is, for two semigroups $F_1$ and $F_2$ acting on $X$ investigate the intersection \eq{mixed}{D(F_1) \cap ND(F_2).} Problems of this type are amenable neither to the full-measure argument, nor to the technique based on Schmidt games. Yet there are %some recent 
{many results 
%establishing the thickness of
where sets of type \equ{mixed} are shown to be uncountable and dense. 
Furthermore, one may also strengthen the density to equidistribution and study the thickness of the intersection \eq{mixed2}{Eq(F_1)\cap ND(F_2),}
where $Eq(F_1)$ is the set of %those 
generic points with respect to some natural measure on $X$ in the sense of Birkhoff's pointwise ergodic theorem. 
For example, Schmidt  proved the uncountability of the set \equ{mixed2} when $F_i$ are expanding endomorphisms of the circle with multiplicatively independent bases \cite{s60}, and then considered a similar problem for ergodic toral automorphisms \cite{s64}; see also \cite{BM, BMP}.} 
More recently, in 2013, Bergelson, Einsiedler and Tseng \cite{bet15} showed that for two commuting hyperbolic automorphisms on a torus, or for two elements in a Cartan action on a compact homogeneous space, if the two semigroups $F_1,F_2$ generated by them have trivial intersection, then the set 
%$D(F_1)\cap ND(F_2)$ 
\equ{mixed} is thick in $X$. Afterwards, Tseng \cite{T17}, Lytle and Maier \cite{LM18} proved similar results for two certain non-commuting toral diffeomorphisms. {This was generalized further by Einsiedler and Maier \cite{EM20} and  Wu \cite{wu20}.}

%In 2020, Einsiedler and Maier \cite{EM20} proved that for two totally irreducible toral automorphisms with different weak stable subspaces, the two semigroups $F_1,F_2$ generated by them satisfy that the set $Eq(F_1)\cap ND(F_2)$ is thick in $X$. Later Wu \cite{wu20} generalized these results to two certain partially hyperbolic diffeomorphisms on a compact manifold.

%\comm{Need some statements from \cite{bet15, shit15} and the Einsiedler--Maier paper, maybe more. To state  the results from the latter paper, introduce the notion of $Eq(F)$ in case $F = (g_t)_{t\in\RR}$. } 
\smallskip

Observe that in all aforementioned results the space $X$ was taken to be compact. 
%The only known result about non-compact spaces was established by Shi and Tseng in \cite{shit15}, where they proved that the set of points with nondense orbits under the $\times n$ map on the circle and with dense orbits under the geodesic flow after identifying a circle with a periodic horospherical orbit in the modular surface, has full Hausdorff dimension. \comm{I am not sure if this reference belongs here, since it is not about two maps on the same space. If it does, we should also mention various later equidistribution results for fractals, such as Simmons--Weiss, Khalil--Luethi etc. Maybe better to omit the previous sentence.}
{In this paper we present the first instance of results of this mixed type for dynamical systems on non-compact spaces. Moreover, we will concentrate on a special mode of non-density, namely by demanding that the corresponding orbit be bounded. Namely we will define $$	B(F):=\{x\in X: Fx \text{ is bounded in }X\},$$
which is a subset of $ND(F)$ if $X$ is not compact,
and will study the intersections $Eq(F_1) \cap B(F_2)\subseteq D(F_1) \cap B(F_2)$
for two commuting actions $F_1,F_2$ on $X$.}

\smallskip

To state our main results we first review some basic knowledge about real semisimple Lie groups and their Lie algebras. {One may refer to \cite{knapp} for more details.} %\comm{Maybe a reference to Knapp here? I suspect the exposition is taken from his book, no?} 
Let $G$ be a real semisimple Lie group, let $\mathfrak{g}$ be its Lie algebra, and let $\mathfrak{c}$ be a Cartan subalgebra of $\mathfrak{g}$. It is invariant under some Cartan involution of $\mathfrak{g}$, which %also 
induces a Cartan decomposition $\mathfrak{g}=\mathfrak{k}\oplus\mathfrak{p}$. Then we have \eq{kp}{
\mathfrak{c}=(\mathfrak{c}\cap\mathfrak{k})\oplus (\mathfrak{c}\cap\mathfrak{p}).
}
The $\ad$-action of $\mathfrak{a}:=\mathfrak{c}\cap\mathfrak{p}$ induces a restricted root space decomposition of $\mathfrak{g}$: $$
\mathfrak{g}=\mathfrak{g}_0\oplus\bigoplus_{\chi\in \Sigma}\mathfrak{g}_\chi,
$$ 
where $\mathfrak{g}_0=Z_{\mathfrak{g}}(\mathfrak{a})\supseteq\mathfrak{c}$, and $\mathfrak{g}_{\chi}$ is the restricted root subspace with respect to $\chi\in\Sigma$. We note that $\Sigma$ is a finite set which spans $\mathfrak{a}^*$. %\comm{(I don't understand this sentence; did you mean $\Sigma$? do you really need $\{0\}$?)}
The kernel of each $\chi\in\Sigma$ defines a hyperplane in $\mathfrak{a}$; the complement of these hyperplanes in $\mathfrak{a}$ is disconnected, and each connected component is called an open Weyl chamber. The closure of an open Weyl chamber is called a closed Weyl chamber. For two nonzero vectors $\bu,\bv$ in $\mathfrak{a}$, we say that they lie in opposite closed Weyl chambers if there exists a closed Weyl chambers of $\mathfrak{a}$ containing $\bu$ and $-\bv$. Similar definitions can be made for two rays in $\mathfrak{a}$. %\comm{I think we also need to have the definition of opposite chambers.}

Using this terminology, we formulate two main theorems of this paper.
%, {which give partial results towards Conjectures \ref{MainConj}-\ref{MultiConj}.} 
{In what follows we will consider dynamical systems on homogeneous spaces $X=G/\Gamma$, where {$G$ is a  Lie group and $\Gamma$ is a lattice in $G$. The  action of a one-parameter subgroup $F=\{g_t:t \in \RR\}$  of $G$ by left translations defines a flow on $X$. For such an $F$ we will denote by $F^+$ the semigroup corresponding to non-negative values of $t$, that is, let $F^+:= \{g_t:t >0\}$. We will say  that $F$ is \textit{$\Ad$-diagonalizable} if each $\Ad(g_t)$ is diagonalizable over $\RR$.}

Our first result
is  an observation that one can construct points with different types of orbit behavior when the acting one-parameter groups, loosely speaking, point in an opposite direction from each other. It} 
%of them 
is a simple consequence of Marstrand's slicing theorem.

\begin{theorem}\label{MainThm00}
    Let $G$ be a real semisimple Lie group,  $\Gamma\subset G$   an irreducible lattice in $G$, and $X=G/\Gamma$. Let $F_1$ and $F_{2,j}\,(j\in\NN)$ be %\new{nontrivial} %\comm{(isn't noncompact automatic now?)} 
    $\Ad$-diagonalizable one-parameter subgroups in a Cartan subgroup of $G$ such that the following two conditions hold:
    \begin{enumerate}
        \item[\rm (1)] for each $j\in\NN$, $\mathrm{Lie}(F_1^+)$ and $\mathrm{Lie}(F_{2,j}^+)$ lie in opposite closed Weyl chambers; 
        \item[\rm (2)] either $\Lie(F_{2,j}^+)\,(j\in\NN)$ are all the same, or all $\Lie(F_{2,j}^+)\,(j\in\NN)$ lie in a common open Weyl chamber.
    \end{enumerate}
    Then the intersection $$
    B(F_1)\cap \bigcap_{j\in\NN}Eq(F_{2,j}^+)
    $$ is thick in $X$.
    %\begin{enumerate}
        %\item[\rm (1)] $B(F_1^+)\cap Eq(F_2^+)$;
        %\item[\rm (2)] $B(F_1^+)\cap B(F_2^+)$.
    %\end{enumerate}
\end{theorem}

%\comm{Two remarks: 1) we can also do the same for $B(F_1)\cap \bigcap_{j\in\NN}Eq(F_{2,j}^+)   $ as long as $F_{2,j}$ are all in the Weyl chamber opposit to $F_1$, why not? 2) the result about $B(F_1^+)\cap B(F_2^+)$ kind of does not belong here, maybe downgrade it to a remark in the main body of the paper?}

%\color{red} Due to technical reasons, we can only prove a special case of Conjecture \ref{MultiConj} at this moment. In the following let us take $d=3,m=n=1$,  and choose $v^{1}=\diag(1,-\lambda_1,-\lambda_2)$ where $0<\lambda_1<\lambda_2$, $\lambda_1+\lambda_2=1$, and $v^2=\diag(1,1,-2)$.\color{black}

In the special case $X_3=\SL_3(\RR)/\SL_3(\ZZ)$ %, we may strengthen the above result by replacing rays with lines and removing the condition on horospherical subgroups.
{we are able to remove restrictions (1) and (2) of Theorem \ref{MainThm00}. Define $A\subset \SL_3(\RR)$ to be the connected diagonal subgroup: 
\eq{def-a}{
A:=\left\{\diag(e^{t_1},e^{t_2},e^{t_3})
:t_i\in\RR,\;\sum_{i=1}^3t_i=0\right\}.
	}
Here is our second main result:}

\begin{theorem}\label{MainThm1}
    Let $X_3=\SL_3(\RR)/\SL_3(\ZZ)$,  {and let} $F_1$ and $F_{2,j}\,(j\in\NN)$ be one-parameter subgroups  {of $A$} %in a split Cartan subgroup of $\SL_3(\RR)$, 
    such that $F_1\neq F_{2,j} \;\forall j\in\NN$. Suppose that  {$F_1$ is regular, that is,} $\mathrm{Lie}(F_1^+)$ is contained in an open Weyl chamber. Then the intersection $$
B(F_1)\cap \bigcap_{j\in\NN}D(F_{2,j}^+)    
    $$
    has full Hausdorff dimension in $X_3$.
\end{theorem}

\begin{remark} Due to some technical reasons (see the proof of Proposition \ref{partialhaar}) we have to impose the assumption that $F_1$ is regular, although the result should probably be true without it.
    Due to other technical reasons (see the remark after Lemma \ref{L:localdim}), we cannot prove that the intersection is thick in $X_3$.
    %in Theorem \ref{MainThm1}. %\comm{We can also mention that due to other technical reasons we have to impose the assumption that $F_1$ is regular, although the result should probably be true without it.}
\end{remark}

\subsection{Applications towards Uniform Littlewood's Conjecture}\label{S:imply}
%\comm{Here I decided to switch around the sections and first talk about some unconditional results.. but is it a good idea? now I am not sure. Your thoughts are welcome!}\new{I can' agree more!}
Theorem \ref{MainThm1} 
%and \ref{Main:oneweyl} have some 
has an application towards a uniform version of Littlewood's conjecture recently introduced by Bandi, Fregoli and the first-named author in \cite{BFK25}. 
%which is the original source of motivations for our paper. To see this we first formulate t
Recall that the classical Littlewood's conjecture can be stated as follows: 
\begin{conjecture}[Littlewood]\label{C:LC}
For any pair of real numbers $(\alpha,\beta)$ and any $\epsilon>0$, there is an unbounded set of $T>0$ such that the %following 
system \begin{equation}\label{littlewood}
		\begin{cases}
		|p+q\alpha||r+q\beta|<\epsilon/T\\
		q\leq T
		\end{cases}
	\end{equation}
has a solution $(p,r,q)\in\ZZ\times\ZZ\times\NN$.
%for an unbounded subset of positive $T$.
\end{conjecture}
One says that the pair $(\alpha,\beta)$ is {\sl multiplicatively well approximable} (MWA) if it satisfies the conclusion of Conjecture \ref{C:LC}, and {\sl multiplicatively badly approximable} (MBA) otherwise. It follows from  {a theorem of Gallagher \cite{Gal62}} %\comm{(reference?)} 
that almost all pairs $(\alpha,\beta)$ are MWA. The best result towards Conjecture \ref{C:LC} was obtained two decades ago by Einsiedler, Katok and Lindenstrauss \cite{EKL06}, who proved that the set of   MBA pairs has zero Hausdorff dimension.

\smallskip
One of the goals of the paper \cite{BFK25} was to study the uniform analogue of multiplicative Diophantine approximation. Namely, a pair $(\alpha,\beta)$ was said to be {\sl multiplicatively singular} (MS) if for any $\epsilon>0$ there exists $T_0 > 0$ such that the  system \eqref{littlewood}
has a nonzero integer solution {for any $T > T_0$}. Clearly MS implies MWA. It was shown in \cite{BFK25} that the set of MS pairs has full Lebesgue measure.  However, the following questions, posed as \cite[Question 4.6]{BFK25}, are wide open: 

\begin{question}\label{bfkquestions} \rm

\begin{itemize}
    \item[(i)]  Is it true that  any  $(\alpha,\beta)$ is MS? an affirmative answer would clearly imply that  Conjecture \ref{C:LC} holds in a stronger form, which is sometimes informally referred to as the `Uniform Littlewood's Conjecture'.
\item[(ii)] Assuming the answer to the previous question is negative, does there exist a pair $(\alpha,\beta)$ that is MWA but not MS? \end{itemize}
\end{question}

\ignore{\begin{conjecture}[Uniform Littlewood]\label{C:ULC}
    For any pair of real numbers $(\alpha,\beta)$ and any $\epsilon>0$, the following system \begin{equation*}
		\begin{cases}
		|p+q\alpha||r+q\beta|<\epsilon/T\\
		q\leq T
		\end{cases}
	\end{equation*}
has solutions $(p,r,q)\in\ZZ\times\ZZ\times\NN$ for all large positive $T$.
\end{conjecture}

It is clear that Conjecture \ref{C:ULC} implies Conjecture \ref{C:LC}.

Two decades ago, Einsiedler, Katok and Lindenstrauss \cite{EKL06} proved that the counterexample set to Conjecture \ref{C:LC} has zero Hausdorff dimension. Recently, Bandi, Fregoli and Kleinbock \cite{BFK25} showed that the set of counterexamples to Conjecture \ref{C:ULC} has zero Lebesgue measure. However, it is still unknown whether these counterexample sets are empty or not.}

%Our paper aims to attack these conjectures 
Following \cite{BFK25}, one can interpret the properties discussed above in a dynamical way. To $(\alpha,\beta)\in\RR^2$  one can associate a point in $X_3=\SL_3(\RR)/\SL_3(\ZZ)$:
$$
x_{\alpha,\beta}:=\begin{pmatrix}
    1 &0 &\alpha\\ 0 &1 &\beta \\ 0 &0 &1 
\end{pmatrix}\cdot\SL_3(\ZZ),
$$
and consider its orbit under left translations by elements of %the connected diagonal subgroup 
$A$ 
%of $\SL_3(\RR)$: 
as in \equ{def-a}.
  Further, consider the following subset of $A$:
\eq{aplus}{A^{+}:=\{\diag(e^t,e^s,e^{-t-s}):t,s\ge0\}.}
%The reason for adding the superscript ${}^+$ here is to distinguish it from the following connected diagonal subgroup $$
%A:=\{\diag(e^t,e^s,e^{-t-s}):t,s\in\RR\}.
%$$
Using Mahler's compactness criterion, it was shown in \cite{EKL06} that $(\alpha,\beta)$ is multiplicatively well apprpoximable if and only if the orbit $A^{+}x_{\alpha,\beta}$ is unbounded. 
%Also it was implicitly proved in \cite{CSD55} that Conjecture \ref{C:LC} can be deduced from the following conjecture, explicitly stated by Margulis in \cite{M97}:
%\begin{conjecture}[Margulis]\label{C:Margulis}
%    Any bounded $A$-orbit in $X_3$ is compact.
%\end{conjecture}
Multiplicatively singular pairs were treated in \cite{BFK25} using similar ideas. Namely, let us consider a partition of $A^{+}$ into fibers 
\eq{ltplus}{L_T^{+}:=\left\{\diag(e^t,e^s,e^{-t-s}):t,s>0,t+s=T\right\},
}
where $T$ runs through all non-negative real numbers. Then say that the orbit $A^{+}x$, where $x\in X_3$, is {\sl fiberwise divergent} if for any compact subset $K$ of $X_3$   there exists $T_0 > 0$ such that $L_T^+x\not\subseteq K$ for all $T > T_0$.  It was proved in  \cite{BFK25} that $(\alpha,\beta)$ is multiplicatively singular if and only if $A^{+}x_{\alpha,\beta}$ is fiberwise divergent. Equivalently, $(\alpha,\beta)$ is not MS if and only if there exists a bounded subset $K\subseteq X_3$ such that the set $\{T>0: L_T^{+}x\subseteq K\}$ is unbounded. Thus the two parts of Question \ref{bfkquestions} can be rephrased as follows: 
\begin{itemize}
    \item[(i)]  Is it true that $A^{+}x_{\alpha,\beta}$ is fiberwise divergent for any $(\alpha,\beta)$?
\item[(ii)] Is it possible to find $(\alpha,\beta)$ such that    $A^{+}x_{\alpha,\beta}$   is unbounded but not fiberwise divergent? \end{itemize}

 Being unable to answer these questions, we would like to put them into a more general context. Let 
    $$
    \mathfrak{a}:=\Lie(A) = \left\{\diag(t_1,t_2,t_3):t_i\in\RR,\;\sum_{i=1}^3t_i=0\right\}.
    $$
    By \textit{a cone in $\mathfrak{a}$} we mean a nonempty convex subset $C\subseteq \mathfrak{a}$ with the property that $$
	\bv\in C\Longrightarrow \forall r>0, r\bv\in C.
	$$
	Its image $A_C$ under the exponential map, which is a subsemigroup of $A$, will be  called \textit{a cone in $A$}. For example  $A^+$ as in \equ{aplus} can be written as $A_{C_0}$, where $C_0 := \{\diag(t,s,-t-s):t,s\ge0\}$.

\ignore{We say that a point $x\in X_3$ has {\sl a bounded $A^{+}$-orbit} if $A^{+}x$ is a bounded subset of $X_3$; say that a point $x\in X_3$ has {\sl a weakly bounded $A^{+}$-orbit} if there exists a bounded subset $K\subseteq X_3$ such that $\{T>0: \exp(L_T^{+})x\subseteq K\}$ is an unbounded set, where $$
L_T^{+}:=\{\diag(t,s,-t-s):t,s>0,t+s=T\}.
$$
Using Mahler's compactness criterion, one easily check that for any real pair $(\alpha,\beta)$:\begin{itemize}
		\item[($\neg$LC)] $(\alpha,\beta)$ violates Littlewood's conjecture $\Longleftrightarrow x_{\alpha,\beta}$ has a bounded $A^{+}$-orbit.
		\item[($\neg$ULC)] $(\alpha,\beta)$ violates uniform Littlewood's conjecture $\Longleftrightarrow x_{\alpha,\beta}$ has a weakly bounded $A^{+}$-orbit.
	\end{itemize}
It is interesting to study the difference between $A^+$-boundedness and $A^+$-weak boundedness.}

%More generally, one may also consider bounded and weakly bounded orbits in $X_3$ under the action of an open cone $C\subseteq\mathfrak{a}$. 
Suppose that we are given a cone $C\subset \mathfrak{a}$\ignore{which is not a half-plane}. Say that a linear functional $\lambda$ on $\mathfrak{a}$ is {\sl compatible with} $C$ if $C\cap \lambda^{-1}(T) $ is bounded and non-empty for any $T > 0$. In this case for any $T > 0$ let us define $$L_{C,\lambda,T}:= A_C \cap \exp\big(\lambda^{-1}(T)\big) =  \left\{\exp(\bv)
:\bv\in C,\;\lambda(\bv) = T\right\}.$$
\ignore{\textit{A family $\cL_C$ of parallel segments in $C$} is defined by the following data: \begin{itemize}
        \item $\cL_C=\{L_{T,C}:T>0\}$ is a collection of segments in $C$ parametrized by $T>0$;
        \item for each $T>0$, the interior of $L_{T,C}$ is contained in $C$, while the two endpoints of $L_{T,C}$ lies on $\partial C$;
        \item for each $T\neq T'>0$, the segments $L_{T,C}$ and $L_{T',C}$ are parallel;
        \item $\dist(0,L_{T,C})$ is a nonconstant linear function in $T$.
    \end{itemize}}
    For example $L^+_T$ as in \equ{ltplus} can be written as $L_{C_0,\lambda_0,T}$, where $\lambda_0(t_1,t_2,t_3) = t_1 + t_2$.
    
   Now, for $x\in X_3$, say that $A_Cx$ is  {{\sl $\lambda$-fiberwise divergent} 
   if for any compact   $K\subset X_3$   there exists $T_0 > 0$ such that $L_{C,\lambda,T}x\not\subseteq K$ for all $T > T_0$, and {\sl $\lambda$-fiberwise non-divergent} otherwise}.
   %there exists a bounded subset $K$ of $X_3$ such that $\{T>0: L_{C,\lambda,T}x\subseteq K\}$ is an unbounded set. Clearly boundedness of $A_Cx$ implies its $\lambda$-fiberwise nondivergence. {Moreover, it follows from    \cite[Theorem 1.4]{BFK25} that for any cone $C$ and any compatible linear functional $\lambda$, the orbits $L_{C,\lambda,T}x$ become equidistributed in $X_3$ as $T\to\infty$ for almost all $x\in X_3$; hence the set of points with $\lambda$-fiberwise divergent orbits has full Haar measure in $X_3$.}
   
   Using the above terminology, one may 
   attempt to pose a version of Question~\ref{bfkquestions}  for all points $x\in X_3$ (not only for points of the form $x_{\alpha,\beta}$) as follows:

\begin{question}\label{Qs:wbb}
    Let $C\subset\mathfrak{a}$ be a cone\ignore{which is not a half-plane}, and let $\lambda$ be compatible with $C$. Then: 
    \begin{itemize}
          \item[\rm (i)]  Can one %somehow 
          describe all $x\in X_3$ with a $\lambda$-fiberwise non-divergent orbit $A_Cx$?
\item[\rm (ii)]  Does there exist $x\in X_3$ such that its  $A_C$-orbit is $\lambda$-fiberwise nondivergent but unbounded?
    \end{itemize}
   \end{question}

%An obvious observation is that for any $x\in X_3$, $$Ax \text { compact } \Longrightarrow A^+x \text { bounded} \Longrightarrow A^+x \text { weakly bounded}.$$ The natural question is: do the converse directions hold? More generally, what if we replace $A^+$ by other cones in $A$?

%Here we give a negative answer to Question \ref{Qs:wbb} for wide enough cones $C\subseteq \mathfrak{a}$ and almost all compatible linear functionals $\lambda$:

%\comm{So here is a crucial point. I think these are very logical questions. Part (i) obviously is as hard as Conjecture \ref{C:Margulis} or its stronger version, and maybe we should discuss it in the appropriate section after introducing the conjectures. But can we say anything interesting and new about part (ii) unconditionally, just using Theorem~\ref{MainThm1}? I thought we could, but now I am not so sure. Your thoughts are welcome! For now I stopped here and simply rewrote the remaining results using the new terminology, just to see how it looks.}

{Here we give a stronger result which implies an affirmative answer to Question~\ref{Qs:wbb}(ii):
\begin{theorem}\label{Main:ans0}
Let $C\subset\mathfrak{a}$ be an open cone\ignore{which is not a half-plane}, and let 
$\lambda$ be a linear functional compatible with $C$ that is not a root. %$\cL_C$ be a family of parallel segments in $C$ which are not parallel to any Weyl walls. 
Then there exists a dense subset of $X_3$ in which each point has a $\lambda$-fiberwise nondivergent $A_C$-orbit and an unbounded $F^+$-orbit for every ray $F^+$ in $A$.
\end{theorem}
We remark here that if $C=C_0$, and if in the dense subset of Theorem \ref{Main:ans0} we were able to find a point $x$ of the form $x_{\alpha,\beta}$, it would solve Question 1.7(ii). However our methods do not allow it. 
}

\subsection{The Structure of the Paper} Our paper is organized as follows. In Section \ref{S:Proof1}, we prove a more general result than Theorem \ref{MainThm00} using Marstrand's slicing theorem. The whole Section \ref{entropy} is devoted to the proof of Theorem \ref{MainThm1}, which is divided into three steps and relies heavily on the entropy arguments.
%We prove Theorems \ref{MainThm00} and \ref{MainThm1}   in Sections \ref{S:Proof1} and \ref{entropy} respectively. 
One ingredient of the proof is a technical lemma, an analogue of \cite[Proposition 2.4]{bet15}, whose proof is relegated to Appendix \ref{A:a}. 

To show the application of our main results towards Question~\ref{Qs:wbb}(ii), in Section \ref{nonconstr} we prove a slightly stronger Baire category variant (see Proposition \ref{P:nondiv}), from which Theorem \ref{Main:ans0} follows. 

 \subsection*{Acknowledgments}
Both authors are grateful to Prasuna Bandi and Reynold Fregoli for discussions during their visit in Brandeis University. The first-named author would like to thank Elon Lindenstrauss and Barak Weiss for helpful suggestions and discussions. The second-named author would also like to thank Jinpeng An, Manfred Einsiedler, Ronggang Shi and Runlin Zhang for their generous help and inspiring discussions. 

\section{Marstrand's Slicing Arguments: Proof of Theorem \ref{MainThm00}}\label{S:Proof1}

%Let $G$ be a Lie group, $\Gamma$ be a lattice in $G$, and let $X=G/\Gamma$. Let $F$ be a one-parameter $\Ad$-non-quasi-unipotent subgroup in $G$ \comm{(probably need to define those)}, and fix a parametrization of $F$ as $\{g_t\}_{t\in\RR}$ or $\{g_t\}_{t\in\ZZ}$. Then the subgroup $F$ acts on $X$ by left translations. 

\subsection{Horospherical subgroups and their properties} \label{horosph} {In this subsection we let $G$ be a real Lie group, $\Gamma$   a lattice in $G$, $X = \ggm$, and $F=\{g_t:t \in \RR\}$   {an $\Ad$-diagonalizable} one-parameter subgroup of $G$.} {Denote $a:= g_1$, 
%\in G$ be the time-$1$ element of $F$ under a fixed parametrization, %\comm{(not clear what it means either -- any element can be a  time-$1$ element under a suitable parametrization)} 
and let $\sigma_a
%(\Ad(a))
\subseteq\RR$ be the set of eigenvalues of $\Ad(a)$.}
%Then the action of $F$ on the homogeneous space $X=G/\Gamma$ by left translations defines a flow. Assume that $F$ is \textit{$\Ad$-diagonalizable}, that is, each $\Ad(g_t)$ is diagonalizable over $\RR$.
%We are going to consider different foliations of $X$ on which the behaviors of $F$ are different. To this end the following notions are introduced. 
{Let $\mathfrak{g}$ denote the Lie algebra of $G$, and for each $\lambda\in{\sigma_a}%\RR
$, let $E_{\lambda}$ be the $\lambda$-eigenspace of $\Ad({a})$:
\eqx{geneigen}{E_\lambda:=\left\{\bv\in \Lg: \Ad({a})\bv=\lambda \bv\right\}.}
Let $\Lh,\Lh^0,\Lh^-$ be the subalgebras of $\Lg$ so that \eqx{horoalg}{\Lh=\Span\{E_{\lambda}: \lambda>1\},\;
\Lh^0=E_1,\;
\Lh^-=\Span\{E_{\lambda}: 0<\lambda<1\},}
%The fact that $g_1$ is $\Ad$-non-quasi-unipotent implies that $\Lh\neq \{0\}$ and $\Lh^-\neq \{0\}$; moreover, $\dim_\RR(\Lh)=\dim_\RR(\Lh^-)$. 
{and let 
\eq{3sbgps}{H=H(F^+),\  H^0=H^0(F^+), \ H^-=H^-(F^+)}  be the corresponding subgroups of $G$; they are called \textit{unstable, neutral and stable horospherical subgroups} with respect to $F^+$.} 
%Note that $H$ is the unstable horospherical subgroup with respect to $g_1$, while $H^-$ is the stable horospherical subgroup with respect to $g_1$. (Recall that a subgroup \new{\eqx{horo}{\left\{h\in G: g^nhg^{-n}\to 1_G\text{\rm \ as }n\to \infty\text{ (resp.\ as }n\to -\infty)\right\}}}
%is called the stable (resp.\ unstable) horospherical subgroup with respect to $g\in G$.)} 

%{Clearly the subalgebras $\Lh,\Lh^0,\Lh^-$ are invariant under $\Ad g_t$, which implies that the subgroups $H,H^0,H^-$ are normalized by $F$. Moreover, a re-parametrization of $F$ by a positive scalar keeps all of them unchanged, while the one by a negative scalar swaps $H$ and $H^-$. }

%\new{Throughout this paper, we say that $F$ is \textit{$\Ad$-semisimple} if $\Ad(g_1)$ is diagonalizable over $\CC$; say that $F$ is \textit{$\Ad$-diagonalizable} if $\Ad(g_1)$ is diagonalizable over $\RR$.}

%\subsection{Proof of Theorem \ref{MainThm00}}\label{S:Sub}
Let us begin the proof of Theorem \ref{MainThm00} with 
the following result about bounded orbits of {points on %the or the unstable horospherical subgroup,
$H$-orbits}, which is a variant of   \cite[Corollary 5.5]{KM96}.

    \begin{lemma}[Corollary 5.5 in \cite{KM96}]\label{L:KM1.5} {Let $G$ be a Lie group, $\Gamma$ a lattice in $G$, $F$ a one-parameter $\Ad$-diagonalizable subgroup of $G$,  $X=G/\Gamma$ and $H=H(F^+)$.} Then for any $x\in X$, the set $$
    %E^+_H(x,\infty):=
    \{h\in H:hx\in B(F^+)\}
    $$
    is thick in $H$.        
    \end{lemma}

{We also have} a basic lemma about {equidistribution of $g_t$-trajectories of points of the form $hx$, where $h$ lies in  the unstable horospherical subgroup. It} is a variant of Birkhoff's pointwise ergodic theorem using Margulis' thickening trick.

\begin{lemma}\label{L:densefull}
{Let $G$, $\Gamma$, $F$, $X$ and $H$ be as in Lemma \ref{L:KM1.5}, and} 
%be the unstable horospherical subgroup with respect to $F^+$, and. 
    %Let $G$ be a Lie group, $\Gamma\leq G$ be a lattice, $F$ be a one-parameter $\mathrm{Ad}$-diagonalizable subgroup of $G$, and 
   % {Let \new{$H=H(F^+)
   % $}}.
    %be the unstable horospherical subgroup with respect to $F^+$. 
    suppose that the left translation by $F^+$ on $X=G/\Gamma$ is ergodic with respect to the Haar measure on $X$. Then for any $x\in X$, the set $$
    S_x:=\left\{h\in H: hx\in Eq(F^+)\right\}
    $$ has full Haar measure in $H$. 
\end{lemma}
\begin{proof}
    %We {let $H^0%=H^0(F^+)
    %$ (resp.\ $H^-$)} %(resp. $H^-=H^-(F^+)$) to 
   % be the neutral (resp.\ stable) horospherical subgroup of $G$ with respect to $F^+$. 
   {With notation as in \equ{3sbgps},} or a fixed point $x\in X$ we choose small open neighborhoods $\Omega_x,\Omega_x^0,\Omega_x^-$ of $1_G$ in $H,H^0,H^-$ respectively, such that {the map} $$
    \pi_x:\Omega_x^-\Omega_x^0\Omega_x\to X,\quad g\mapsto gx
    $$ is an isometry onto its image. In particular, $\Sigma_x:=\Omega_x^-\Omega_x^0\Omega_x\cdot x$ is an open neighborhood of $x$ in $X$. It follows from the ergodicity of $F^+$ on $X$ that $\Sigma_x\cap Eq(F^+)$ has full Haar measure in $\Sigma_x$. Then its preimage under $\pi_x$ also has full Haar measure in $\Omega_x^-\Omega_x^0\Omega_x$. We note that, {in view of the $\Ad$-diagonalizability assumption on $F$,} this preimage has the product structure $\Omega_x^-\Omega_x^0\cdot(\Omega_x\cap S_x).$ Then by the locally almost product structure of Haar measure, the set $\Omega_x\cap S_x$ has full Haar measure in $\Omega_x$. It is clear that this conclusion also holds for any nonempty open subset of $\Omega_x$.

    Now for any $h'\in H$, we choose a small open neighborhood $\Omega'$ of $h'$ in $H$ such that $\Omega'(h')^{-1}\subseteq \Omega_{h'x}$. Then it follows from the above paragraph that $$\Omega'\cap S_x=(\Omega'(h')^{-1}\cap S_{h'x})h'$$ has full Haar measure in $\Omega'$. Since $H$ has a countable topological basis, we conclude that the set $S_x$ has full Haar measure in $H$. 
\end{proof}

%We also recall 

\subsection{A generalization of Theorem \ref{MainThm00}} \label{general}    
%\new{In this subsection} we prove a more general result than Theorem \ref{MainThm00}, which can be reduced to above two lemmas by using Marstrand's slicing theorem repeatedly. 
For {a family of} one-parameter $\Ad$-diagonalizable subgroups $F_1$ and $F_{2,j}\,(j\in\NN)$ in $G$, we {let} $$H=H(F_1^+)\ (\text{resp.\ } H^0=H^0(F_1^+), \ H^-=H^-(F_1^+))$$ be the unstable (resp.\ neutral, stable) horospherical subgroup of $G$ with respect to $F_1^+$, and let $$W_j=H(F_{2,j}^+)\ (\text{resp.\ } W_j^0=H^0(F_{2,j}^+), \ W_j^-=H^-(F_{2,j}^+))$$  be the unstable (resp.\ neutral, stable) %horospherical 
subgroups %of $G$ 
with respect to $F_{2,j}^+$.

\begin{proposition}\label{P:Thm0}
    Let $G$ be a Lie group, let $\Gamma$ be a lattice in $G$, and let $X=G/\Gamma$ be the corresponding homogeneous space. Let $F_1$ and $F_{2,j}\,(j\in\NN)$ be a commuting family of one-parameter $\mathrm{Ad}$-diagonalizable subgroups in $G$  such that 
 {\begin{enumerate}
        \item[\rm (1)] $W_j,W_j^0,W_j^-$ are all independent of $j\in\NN$, denoted by $W,W^0,W^-$;
        \item[\rm (2)] $W\subseteq H^-H^0$, and $W\cdot (W^0\cap H^-H^0)\cdot (W^-\cap H^-H^0)$ contains an open neighborhood of $1_G$ in $H^-H^0$.
    \end{enumerate}
    }
    Suppose that the left translations by $F_1$ and $F_{2,j}\,(j\in\NN)$ on $X$ are all ergodic with respect to the Haar measure on $X$. Then the intersection 
    $$
    B(F_1^+)\cap \bigcap_{j\in\NN}Eq(F_{2,j}^+)
    $$
    is thick in $X$.
\end{proposition}

\begin{proof}
    For a fixed point $x\in X$, we choose small open neighborhoods $\Omega,\Omega^0,\Omega^-$ of $1_G$ in $H,H^0,H^-$ respectively, such that $$
    \Omega^-\Omega^0\Omega\to X,\quad g\mapsto gx
    $$ is a bi-Lipschitz diffeomorphism onto its image. It suffices to show that $$
    S:=\{g\in \Omega^-\Omega^0\Omega: gx\in B(F_1^+)\cap D\}
    $$
    has full Hausdorff dimension in $\Omega^-\Omega^0\Omega$, where $D:=\bigcap_{j\in\NN}Eq(F_{2,j}^+)$.

    %We first claim that the set $$
    %C:=\{h\in \Omega^0\Omega: hx\in B(\{g^1_t\}_{t\geq 0})\}
    %$$ has full Hausdorff dimension in $\Omega^0\Omega$. In fact, the claim follows from Marstrand slicing theorem, since $C$ has the product structure $\Omega^0\cdot \{h\in\Omega:hx\in B(\{g^1_t\}_{t\geq 0})\}$, and
    
    We note that the set $E:=\{h\in\Omega:hx\in B(F_1^+)\}$ has full Hausdorff dimension in $\Omega$ by Lemma \ref{L:KM1.5}. For each $h\in E$, the slice of $S$ is \begin{equation*}
        \begin{aligned}
            S\cap \Omega^-\Omega^0h&=\{h^-h^0\in \Omega^-\Omega^0:h^-h^0hx\in B(F_1^+)\cap D\}\cdot h\\
            &=\{h^-h^0\in \Omega^-\Omega^0:h^-h^0hx\in D\}\cdot h.
        \end{aligned}
    \end{equation*}
    By Marstrand's slicing theorem (see \cite[Lemma 1.4]{KM96}),
    %\comm{(I think here is where you need to give a reference. You decide to do it after the proof of Corollary \ref{countfulldense} instead, which is strange.)}--good idea
    it suffices to show that for $y=hx\in X$, the set $$
    S':=\{h^-h^0\in \Omega^-\Omega^0:h^-h^0y\in D\}
    $$
    has full Hausdorff dimension in $\Omega^-\Omega^0$.

    It follows from our assumptions
    %\comm{(this requires an explanation -- why can't $W$ be disjoint from $H^-H^0$ as well?)}
    that we can choose small neighborhoods $\Xi,\Xi^0,\Xi^-$ of $1_G$ in $W,W^0\cap H^-H^0,W^-\cap H^-H^0$ respectively, such that $$
    \Xi^-\times \Xi^0\times \Xi\to \Omega^-\Omega^0,\quad (n^-,n^0,n)\mapsto n^-n^0n
    $$
    is a bi-Lipschitz diffeomorphism onto its image. The preimage of $S'$ under this map has the product structure $\Xi^-\times \Xi^0\times\{n\in\Xi: ny\in D\}$, which has full Hausdorff dimension in $\Xi^-\times \Xi^0\times \Xi$ by Lemma \ref{L:densefull} and Marstrand's slicing theorem. Thus $S'$ also has full Hausdorff dimension in $\Omega^-\Omega^0$. This completes the proof. 
\end{proof}

\begin{remark}
    One may loosen some restrictions on the horospherical subgroups in Proposition \ref{P:Thm0} when {proving} analogues of Lemmas \ref{L:KM1.5} and \ref{L:densefull} for a suitable subgroup of the whole unstable horospherical subgroup. See \cite{shi20} for more discussions around this.
\end{remark}

\subsection{Checking the assumptions of Proposition \ref{P:Thm0}} {Now we are going to find natural restrictions on $F_1$ and $F_{2,j}\,(j\in\NN)$ so that {conditions} (1) and (2) in Proposition \ref{P:Thm0} hold. The next two lemmas are simple results about rays in opposite Weyl chambers and their unstable, neutral and stable horospherical subgroups.
\begin{lemma}\label{L:hind}
    Let $G$ be a real semisimple Lie group, and let $F_j\,(j\in\NN)$ be $\Ad$-diagonalizable one-parameter subgroups in a Cartan subgroup of $G$ such that one of the followings holds:
    \begin{enumerate}
        \item[\rm (1)]  $\Lie(F_j^+)\,(j\in\NN)$ are all the same;
        \item[\rm (2)]  all $\Lie(F_j^+)\,(j\in\NN)$ lie in a common open Weyl chamber.
    \end{enumerate}
    Then $H(F_j^+)$, $H^0(F_j^+)$, $H^-(F_j^+)$ are all independent of $j\in\NN$.
\end{lemma}
\begin{proof}
    The case (1) is trivial. For the case (2), we see that $H(F_j^+)$ is the connected subgroup corresponding to the subalgebra  $$
    \Span\{E^{(j)}_{\lambda}:\lambda>1\},
    $$
    where $E^{(j)}_{\lambda}$ is the $\lambda$-eigenspace of $\Ad(g_1^{(j)})$ for $F_j^+=\{g^{(j)}_{t}:t>0\}$. Write $\mathfrak{a}:=\mathfrak{c}\cap\mathfrak{p}$. An open Weyl chamber of $\mathfrak{a}$ is given by $$
    \{\bv\in\mathfrak{a}:\mathrm{sgn}(\chi(\bv))=\epsilon(\chi),\;\forall \chi\in \Sigma^+\},
    $$
    where $\epsilon(\Sigma^+)\subseteq\{\pm 1\}$. Then it follows that \begin{equation*}
        \begin{aligned}
            \Span\{E^{(j)}_{\lambda}:\lambda>1\}
            &=\Span\{\mathfrak{g}_{\chi}:\chi(\log(g^{(j)}_{1}))>0\}\\
            &=\Span\{\mathfrak{g}_{\chi}:\chi(\log(g^{(j')}_{1}))>0\}
            =\Span\{E^{(j')}_{\lambda}:\lambda>1\},
        \end{aligned}
    \end{equation*}
    which implies that $H(F_j^+)=H(F_{j'}^+)$. Similar arguments work for $H^0(F_j^+)$ and $H^-(F_j^+)$.
\end{proof}
\begin{lemma}\label{L:hopen}
    Let $G$ be a real semisimple Lie group, {and let} $F_1$ and $F_2$ be $\Ad$-diagonalizable one-parameter subgroups in a Cartan subgroup of $G$ such that $\Lie(F_1^+)$ and $\Lie(F_2^+)$ lie in opposite closed Weyl chambers. Then $W\subseteq H^-H^0$, and {the set} $$W\cdot (W^0\cap H^-H^0)\cdot (W^-\cap H^-H^0)$$ contains an open neighborhood of $1_G$ in $H^-H^0$. %Additionally, if the image of $\Lie(F_2^+)$ lies in an open Weyl chamber, then $W^0\subseteq H^0$.
\end{lemma}
\begin{proof}
By the same arguments as in the proof of Lemma \ref{L:hind}, we see that \begin{equation*}
    \begin{aligned}
        \Lie(W)&=\Span\{\mathfrak{g}_{\chi}:\chi(\log(g^{(2)}_1))>0\},\\
        &=\Span\{\mathfrak{g}_{\chi}:\chi(\log(g^{(2)}_1))>0,\chi(\log(g^{(1)}_1))\leq 0\}=\Lie(W\cap H^-H^0),\\
    \end{aligned}
\end{equation*}
since $\log(g^{(1)}_1)$ and $\log(g^{(2)}_1)$ lie in opposite closed Weyl chambers. It follows that $W\subseteq H^-H^0$. Similarly, we have \begin{equation*}
    \begin{aligned}
        \Lie(W^0\cap H^-H^0)
        &=\mathfrak{g}_0\oplus\Span\{\mathfrak{g}_{\chi}:\chi(\log(g^{(2)}_1))=0,\chi(\log(g^{(1)}_1))\leq 0\},\\
        \Lie(W^-\cap H^-H^0)
        &=\Span\{\mathfrak{g}_{\chi}:\chi(\log(g^{(2)}_1))< 0,\chi(\log(g^{(1)}_1))\leq 0\}.\\
    \end{aligned}
\end{equation*}
It follows that $$
\Lie(W)\oplus \Lie(W^0\cap H^-H^0)\oplus \Lie(W^-\cap H^-H^0)=\Lie(H^-H^0).
$$
This implies that $W\cdot (W^0\cap H^-H^0)\cdot (W^-\cap H^-H^0)$ contains an open neighborhood of $1_G$ in $H^-H^0$. %Moreover, if $\log(g^{(2)}_1)$ lies in an open Weyl chamber, then $$\Lie(W^0)=\mathfrak{g}_0\subseteq \Lie(H^0),$$ which implies that $W_0\subseteq H^0$.
\end{proof}
}

Finally, in view of Moore's ergodicity theorem {and the assumptions on $G$ and $\Gamma$} we immediately deduce Theorem \ref{MainThm00} as a corollary of Proposition \ref{P:Thm0}:

\begin{proof}[Proof of Theorem \ref{MainThm00}] %Since $F_1$ and $F_{2,j}\,(j\in\NN)$ are all noncompact $\Ad$-diagonalizable, %\comm{(now it becomes important –- over $\RR$ or over $\CC$?)} 
%we have $\Lie(F_1^+),\Lie(F_{2,j}^+)\subseteq \mathfrak{c}\cap \mathfrak{p}$. 
In view of Lemmas \ref{L:hind} and \ref{L:hopen}, the assumptions (1) and (2) \new{of} Proposition \ref{P:Thm0} hold. Moreover, by Moore's ergodicity theorem we see that the left translations by $F_1$ and $F_{2,j}\,(j\in\NN)$ are all ergodic with respect to the Haar measure on $X$. Then the conclusion follows from Proposition \ref{P:Thm0}.
   % \comm{I think this should be explained too, using the definition of opposite chambers.} 
\end{proof}

\section{Entropy Arguments: Proof of Theorem \ref{MainThm1}} \label{entropy} For Theorem \ref{MainThm1}, our method of proof is similar to that in \cite{John95}, where an average version of Host's theorem is established. The whole proof is divided into three steps.

\subsection{Step 1: High entropy arguments} In this subsection we introduce a general background of high entropy arguments. Let {$G,\Gamma,X,F,a,\sigma_a,H,H^0,H^-$ be as in Section \ref{horosph}}. %\comm{(not clear what it means -- non-quasi-unipotent? semisimple? diagonalizable?)} 
%$a\in G$ be the time-$1$ element of $F$ under a fixed parametrization, %\comm{(not clear what it means either -- any element can be a  time-$1$ element under a suitable parametrization)} 
%and $\new{\sigma_a}\subseteq\RR$ be the set of eigenvalues of $\Ad(a)$. 
It is well-known that $$
h_{top}(a)
=\sum\limits_{\lambda\in{\sigma_a}}\log^+\lambda\cdot\dim_\RR(E_{\lambda}),
$$
where $E_{\lambda}$ is the $\lambda$-eigenspace of $\Ad(a)$. The %first 
{next} lemma calculates the topological entropy of $a$ restricted to an invariant compact set via its Hausdorff dimension. Its proof is similar to that of \cite[Proposition 2.4]{bet15}, {which established an analogous formula for hyperbolic toral automorphisms.}

%A \text{\sl higher rank Cartan action} on $X$ is given by an injective group homomorphism from $\ZZ^{k_1}\times \RR^{k_2}\times\TT^{k_3}$ to a Cartan subgroup of $G$ with $k_1,k_2,k_3\in\NN\cup\{0\}$ and $k_1+k_2+k_3\geq 2$, which acts by left translations on $X$. 

%Assume that $\alpha:\ZZ^2\to G$ is the parametrization of a subgroup of a Cartan subgroup of $G$, which projects injectively and discretely to each simple factor. We will identify $\alpha$ also with the induced action on the left of $X=G/\Gamma$. We write $\alpha^{\vec{t}}$ for the action of an individual element of a two-dimensional subgroup of the Cartan subgroup where $\vec{t}\in\ZZ^2$.

%\begin{lemma}[Measure Rigidity]
%	Let $G,\Gamma$ and $\alpha$ be as above, and $\vec{t}\neq 0$ be a fixed element. If $\mu$ is an $\alpha$-invariant ergodic probability measure on $G/\Gamma$ with $h_\mu(\alpha^{\vec{t}})$ close enough to $h_{top}(\alpha^{\vec{t}})$, then $\mu$ is the Haar measure on $X$. 
%\end{lemma}
%\begin{proof}   See \cite{EK05}, Theorem 2.4.\end{proof}

%Let $\Lg$ be the Lie algebra of $G$, and $\Phi$ be the root system with respect to the above Cartan subgroup; for any $\lambda\in \Phi$, let $\mathfrak{g}_{\lambda}$ be the corresponding root subspace. Write the above Cartan action as $a:\ZZ^{k_1}\times \RR^{k_2}\times\TT^{k_3}\to A$ and $\alpha=\log\circ a$. Then for any $\bv\in \ZZ^{k_1}\times \RR^{k_2}\times\TT^{k_3}$, the adjoint action of $a(\bv)$ on $\mathfrak{g}_{\lambda}$ is multiplication by the eigenvalue $e^{\lambda(\alpha(\bv))}$. 

\begin{lemma}\label{L:entropydim}
	Let $K\subseteq X$ be a compact $a$-invariant set, and {let} $|\lambda_1|$ be the largest absolute value of eigenvalues of $\Ad(a)$. Then we have $$
	h_{top}(a\big|_K)\geq h_{top}(a)-(\dim X-\dim K)\cdot\log|\lambda_1|.
	$$
\end{lemma}
\begin{proof}
    See Appendix \ref{A:a}.
\end{proof}

From now on we fix a compact exhaustion $\{K_\delta\}_{\delta>0}$ of $X$.

\begin{corollary}\label{C:entropydim}
    For any $\epsilon>0$, there exists some $\delta=\delta(\epsilon,a)>0$ such that the restriction of $a$ to \begin{equation*}
        B(F,K_\delta):=\{x\in X: Fx\subseteq K_\delta\}
    \end{equation*} satisfies $h_{top}(a|_{B(F,K_{\delta})})>h_{top}(a)-\epsilon$.
\end{corollary}
\begin{proof}
Since by \cite[Theorem 1.5]{KM96}  the set $B(F)=\bigcup_{\delta>0}B(F,K_\delta)$ has full Hausdorff dimension in $X$, we may choose $\delta=\delta(\epsilon,a)>0$ small enough such that  
$$
\dim(B(F,K_\delta))>\dim(X)-\dfrac{\epsilon}{\log|\lambda_1|}.
$$
Then the conclusion follows from Lemma \ref{L:entropydim}.
\end{proof}

{Next we recall some basics about Ledrappier--Young's formula for the measure-theoretic entropy. Readers may refer to \cite{AB16} for more necessary definitions. For the left translation by $a$ on $X$, we list the absolute values of eigenvalues of $\Ad(a)$ that are bigger than $1$ as follows: $$
|\lambda_1|>\cdots>|\lambda_k|>1.
$$
For each $i\in\{1,\dots,k\}$, we write $\cW^i$ for the $i$-th unstable foliation with respect to the $a$-action, and let $\xi^i$ be a measurable partition subordinate to $\cW^i$. Now let $\nu$ be an $a$-invariant, ergodic probability measure on $X$,}
{and let $\{\nu^{\xi^i}_x:x\in X\}$ be a family of conditional measures relative to the measurable partition $\xi^i$.}
{It was proved in \cite[Sections 7.3 and 10.1]{AB16} that the limit 
$$
\dim^i(\nu,x):=\lim_{r\to 0^+}\frac{\log\nu^{\xi^i}_x\big(B(x,r)\big)}{\log{r}}, 
$$
exists and equals a constant almost everywhere. It is called
%\comm{(what is $\log^{\xi^i}_\nu$? and do you know where exactly it was proved? \cite[Prop.\ 7.4]{AB16} simply states this as a fact...)}\neww{The proof of Prop 7.4 is in the Section 10 of \cite{AB16}.}
%which is called 
the \textit{$i$-th pointwise dimension of $\nu$}, denoted by $\dim^i(\nu)$. Set $\dim^0(\nu)=0$. Then for $i\in\{1,\dots,k\}$, the \textit{$i$-th transverse dimension of $\nu$} is defined to be $$
\gamma_i(\nu):=\dim^i(\nu)-\dim^{i-1}(\nu).
$$
It is clear that for each $i\in\{1,\dots,k\}$, $$\gamma_i(\nu)\leq \sum\limits_{|\lambda|=|\lambda_i|}\dim(E_{\lambda}).$$ The Ledrappier--Young formula (see \cite[Theorem 7.7]{AB16}) states that $$
h_{\nu}(a)=\sum\limits_{i=1}^k\log|\lambda_i|\cdot \gamma_i(\nu).
    $$}
{Moreover, for each $i\in\{1,\dots,k\}$, we define the entropy contribution of the $i$-th transversal direction to be $$D^i_{\nu}(a):=\log|\lambda_i|\cdot \gamma_i(\nu).$$  We also write the unstable dimension of $\nu$ as $d^+(\nu):=\sum\limits_{i=1}^k\gamma_i(\nu)$. By replacing $a$ above by $a^{-1}$, one may similarly define the stable dimension $d^-(\nu)$.}

%For any $a$-invariant, ergodic probability measure $\nu$ on $X$, we recall the Ledrappier--Young formula (see \cite[Theorem 7.7]{AB16}): $$h_{\nu}(a)=\sum\limits_{i=1}^k\log|\lambda_i|\cdot \gamma_i(\nu),$$
   % where $|\lambda_1|>\cdots>|\lambda_k|>1$ are all  absolute values of eigenvalues of $\Ad(a)$ {that are bigger than $1$}, and $\gamma_i(\nu)$ denotes the $i$-th transversal dimension of $\nu$. It is clear that for each $1\leq i\leq k$, $$\gamma_i(\nu)\leq \sum\limits_{|\lambda|=|\lambda_i|}\dim(E_{\lambda}).$$ Moreover, the entropy contribution of $a$ with respect to $\nu$ along the $i$-th transversal direction is $D_{\nu}(a,U_i)=\log|\lambda_i|\cdot \gamma_i(\nu)$.  We write the unstable dimension of $\nu$ as $d^+(\nu):=\sum\limits_{i=1}^k\gamma_i(\nu)$. By replacing $a$ above by $a^{-1}$, one may similarly define the stable dimension $d^-(\nu)$.

\begin{lemma}\label{L:dimclose}
For any $\eta>0$ there exists some $\epsilon=\epsilon(\eta,a)>0$ such that for any $a$-invariant, ergodic probability measure $\nu$ on $X$ with $h_{\nu}(a)>h_{top}(a)-\epsilon$, one has $d^+(\nu)\geq \dim(H)-\eta$, $d^-(\nu)\geq \dim(H^-)-\eta$, and $D^1_{\nu}(a)>0$.
\end{lemma}
\begin{proof}
    For a given $\eta>0$ we set $$\epsilon:=\min\left\{\eta\cdot\left(\sum\limits_{i=1}^k\frac{1}{\log|\lambda_i|}\right)^{-1},\;
    \log|\lambda_1|\cdot\sum\limits_{|\lambda|=|\lambda_1|}\dim(E_{\lambda})
    \right\}>0.$$ 
    Since $h_{\nu}(a)>h_{top}(a)-\epsilon$, we have for each $1\le i\leq k$, $$\gamma_i(\nu)>\sum\limits_{|\lambda|=|\lambda_i|}\dim(E_{\lambda})-\dfrac{\epsilon}{\log|\lambda_i|},$$ and hence that
    $$
    d^+(\nu)=\sum\limits_{i=1}^k\gamma_i(\nu)    >\sum\limits_{i=1}^k\sum\limits_{|\lambda|=|\lambda_i|}\dim(E_{\lambda})-\sum\limits_{i=1}^k\frac{\epsilon}{\log|\lambda_i|}\geq\dim(H)-\eta.
    $$
    A similar calculation yields that $d^-(\nu)\geq \dim(H^-)-\eta$. Moreover, we see that \begin{equation*}
        \begin{aligned}
            D^1_{\nu}(a)&\geq
            h_{\nu}(a)-\sum_{i=2}^k\log|\lambda_i|\cdot\sum\limits_{|\lambda|=
            |\lambda_i|}\dim(E_{\lambda})\\
            &>h_{top}(a)-\epsilon-\sum_{i=2}^k\log|\lambda_i|\cdot\sum\limits_{|\lambda|=|\lambda_i|}\dim(E_{\lambda})
            \geq 0.
        \end{aligned}
    \end{equation*} 
  This completes the proof.  
    \end{proof}

Here we also give a geometric interpretation for the unstable and stable dimensions of a probability measure.

\begin{lemma}\label{L:localdim}
For any probability measure $\nu$ on $X$ and  any measurable set $C\subseteq X$ with $\nu(C)=1$ 
%there exists a measurable set $C'\subseteq X$ with $\nu(C')=1$, such that for any nonempty open subsets $V,V^-$ in $H,H^-$ respectively and any $x\in C'$, 
one has $\dim(C\cap Hx)\geq d^+(\nu)$ and $\dim(C\cap H^-x)\geq d^-(\nu)$ for $\nu$-a.e. $x\in X$.
\end{lemma}
\begin{proof}
Let $\nu_x$ denote the conditional measure of $\nu$ along the unstable foliation $Hx$ for $\nu$-a.e. $x\in X$. Given any measurable set $C\subseteq X$ with $\nu(C)=1$, it follows that $\nu_x(Hx\setminus C)=0$ for $\nu$-a.e. $x\in X$. 
        We note that the unstable dimension $d^+(\nu)$ can be reinterpreted as the pointwise dimension of $\nu_x$ for $\nu$-a.e. $x\in X$. Then %for any nonempty open subset $V\subseteq H$, 
        applying the mass distribution principle gives $$\dim(C\cap Hx)\geq d^+(\nu) $$ for $\nu$-a.e. $x\in X$. A similar argument for $a^{-1}$ and $H^-$ gives $\dim(C\cap H^-x)\geq d^-(\nu)$ for $\nu$-a.e. $x\in X$.
        %Let $C'(V,V^-)$ be the $\nu$-full subset of $X$ consisting of those $x$ with $\dim(C\cap Vx)\geq d^+(\nu)$ and $\dim(C\cap V^-x)\geq d^-(\nu)$. Take any countable topological basis $\{V_i\}_{i\in\NN},\{V_j^-\}_{j\in\NN}$ of $H,H^-$ respectively, and write $C':=\bigcap_{i,j\in\NN}C'(V_i,V_j^-)$. Then $C'$ is a $\nu$-full subset in $X$ satisfying the desired property.
\end{proof}

It should be remarked that one cannot replace $H,H^-$ by their nonempty open subsets in Lemma \ref{L:localdim}. This is because the conditional measure $\nu_x$ may be supported on a proper fractal inside the unstable manifold $Hx$.

\subsection{Step 2: Averaging measures} In this subsection we study some properties of the limit measure along a sequence of $a$-ergodic measures with large entropy. Let $G,\Gamma,X$ {be} as in Section \ref{horosph}, {let} $F_1,F_2$ be two commuting $\Ad$-diagonalizable subgroups of $G$, $H=H(F_1^+)$, and {let} $a$ be the time-$1$ element of $F_1$ under a fixed parametrization.

Let us fix any $\eta>0$ and choose $\epsilon=\epsilon(\eta,a)$ as in Lemma \ref{L:dimclose}. Then, by Corollary \ref{C:entropydim}, there exists some $\delta=\delta(\epsilon, a)$ such that $$
h_{top}(a|_{B(F_1,K_{\delta})})>h_{top}(a)-\epsilon.
$$
It follows from the variational principle of entropy (see \cite[Theorem 8.6]{wal00}) %\comm{(reference?)} 
that there exists an $a$-invariant Borel probability measure $\nu$ supported on $B(F_1,K_\epsilon)$, such that $$
h_{\nu}(a|_{B(F_1,K_\delta)})>
h_{top}(a)-\epsilon.
$$
By the upper semi-continuity and convexity of measure-theoretic entropy, we may assume that $\nu$ is $F_1$-invariant and $a$-ergodic.

Let $F_2^+=\{g^2_t\}_{t>0}$, and {let} $\mu$ be the weak-* limit along a subsequence of $T\to +\infty$ of the averaging measures:
\eq{E:ave}{\mu_T:=\frac{1}{T}\int_{0}^T(g^2_t)_*\nu\mathrm{d}t.}
Any such limit measure is $\la F_1,F_2\ra$-invariant, but might lose all its mass at infinity. The following proposition solves this problem when $X=X_d:=\SL_d(\RR)/\SL_d(\ZZ)$ and $${a\in A:=\{\diag(a_1,\dots,a_d):a_1,\dots,a_d>0,a_1\cdots a_d=1\}.}$$ Following \cite{ron23}, let us define the ``entropy in the cusp" of $a$ as follows: \begin{equation*}
        \begin{aligned}
h_{\infty}(a):=\sup\{\limsup_{i\to +\infty}h_{\mu_i}(a): \text{each }\mu_i\text{ is an } a\text{-invariant probability measure on }X_d,\\
    \text{and } \mu_i\to 0\text{ in the weak-* topology as } i\to +\infty\}.
        \end{aligned}
    \end{equation*}

{\begin{lemma}\label{L:Ron}
    Let $a=\diag(e^{t_1},\dots,e^{t_d})\in A$. %We write 
    Then $$h_{\infty}(a)\leq h_{top}(a)-\sum_{i=1}^d\max\{0,t_i\}.$$
\end{lemma}
\begin{proof}
    See \cite[Theorem 1.3]{ron23}.
\end{proof}
}

\begin{proposition}\label{nonzero}
	Let $a=\diag(e^{t_1},\dots,e^{t_d})\in A$, and let $\nu$ be an $a$-invariant probability measure on $X_d$ with $$h_{\nu}(a)>h_{top}(a)-\sum_{i=1}^d\max\{0,t_i\}.$$ Then any weak-* limit $\mu$ of 
	$\mu_T$ as $T\to +\infty$ is nonzero.
\end{proposition}
\begin{proof}
    Since $a\in F_1$ commutes with $F_2$, by the convexity of measure-theoretic entropy, we see that for any $T>0$, $h_{\mu_T}(a)=h_{\nu}(a)>h_{top}(a)-\sum_{i=1}^d\max\{0,t_i\}.$ 
    {Now let $\mu$ be a weak-* limit of $\mu_{T_i}$ along a sequence $T_i\to +\infty$. Suppose that $\mu=0$. Then it follows from Lemma \ref{L:Ron} that $\limsup_{i\to +\infty}h_{\mu_{T_i}}(a)\leq h_{\infty}(a)$. This is a contradiction.}
    %Below are your original comments:\new{Then the conclusion follows from \cite[Theorem 1.3]{ron23}.} \comm{Sorry, I still don't understand. \cite[Theorem 1.3]{ron23} involves $h_\infty(a)$ which we are not discussing. I think we should state \cite[Theorem 1.3]{ron23} and explain how it implies the conclusion. Or am I missing something obvious?}\neww{Thanks, let me do it, and then you could decide whether this implication is straightforward or not.}
\end{proof}

{When the limit measure of \equ{E:ave} is nonzero,} %Under this nonzero assumption \comm{(this is not clear, which nonzero assumption do you mean?)}, 
we are able to show that taking limit of averages preserves the positive entropy contribution along some transversal direction. 
For that, let us %To be specific, we 
review here some basic knowledge about leafwise measures along foliations for any $a$-invariant finite positive measure $\nu$ on $X$. More details can be found in \cite{EL10}. Let $U\leq H$ be a closed subgroup normalized by $a$.
Then, following \cite[Section 6]{EL10}, one can define  a system $\{\nu^U_x\}$ of Radon measures on $U$ which we will call the \textit{leaf-wise measures along $U$-foliations}; those  are determined uniquely up to proportionality and outside a set of measure zero.
%and let $\cM(U)$ be the set of Radon measures on $U$. There is a measurable map $X\to \cM(U)$ and the image of $x\in X$ is denoted by $\nu^U_x$. \comm{This is written in a strange way; is it a definition of $\nu^U_x$? then it looks like it depends on the choice of this map $X\to \cM(U)$ which can be taken arbitrarily. Or perhaps you want to define this map explicitly?}
%\neww{The definition of $\nu^U_x$ is quite complicated in \cite{EL10}. I suggest that we just admit its existence and quote some of its properties from \cite{EL10} when necessary. Maybe you have better way to state its existence here.}
%The measure $\nu^U_x$ is called the \textit{leaf-wise measure along $U$ foliations} and is 
We are going to normalize them so that $\nu^U_x(B^U_1)=1$. The \textit{entropy contribution of $U$ with respect to $\nu$} is an $a$-invariant measurable function on $X$ defined by $$
D_\nu(a,U)(x):=\lim\limits_{n\to +\infty}\frac{\log\nu^U_x(a^nB^U_1a^{-n})}{n}.
$$
When $\nu$ is $a$-ergodic, we see that $D_\nu(a,U)(\cdot)$ is constant almost everywhere, denoted by $D_\nu(a,U)$. In particular, when $U$ is the fastest Lyapunov subgroup, %\comm{(what is it? we have not defined those)} \neww{This is 
that is, the Lyapunov subgroup corresponding to the highest weight of $\Ad(a)$, we have $D_\nu(a,U)=D^1_\nu(a)$.

\begin{lemma}\label{L:posentropy}
    Let $a\in A$,   $F_2=\{g_t^2\}_{t\in\RR}\subseteq A$, %\comm{(this notation is confusing, previously we had $a = g_1$ but apparently now we don't, not sure how to handle this)} 
    $U\leq H$  a one-dimensional subgroup normalized by $a$ and $F_2$, and $\nu$   an $a$-invariant, ergodic probability measure on $X_{\neww{d}}$ %\comm{(so this is a theorem about $d=3$ already?)}\neww{--no} 
    with $D_{\nu}(a,U)>0$. If along a subsequence of $T\to +\infty$ the limit measure $\mu$ of $$\mu_T=\frac{1}{T}\int_0^T(g^2_t)_*\nu\mathrm{d}t$$ is nonzero, then $D_{\mu}(a,U)(\cdot)>0$ $\mu$-almost surely.
\end{lemma}
\begin{proof}
    The proof goes the same way as \cite[Theorem 3.1]{shi12}.
    %\comm{Maybe you are right, but it is not convincing to me. I have not read \cite{shi12} before. Ronggang's proof is 4 pages long and it is not obvious to me that it works for arbitrary $a_t$. Are you saying that in \cite[Theorem 3.1]{shi12} $a_t$ can be completely arbitrary? and it works for any $d$, not just $d = 3$?}\neww{--what I mean is that: 
    Note that in \cite[Theorem 3.1]{shi12} it is assumed that $g^2_t=\diag(e^t,\cdots,e^t,e^{-(d-1)t})$, but the proof goes through similarly for any other one-parameter subgroup of $A$. %\comm{I am not very happy with such a remark but if you think it's OK let's keep it; feel free to edit if you find a better way.}
    %can be removed, but other conditions on $a_t$ must be kept as above.--Yes, I changed $X_3$ to $X_d$ above. 
\end{proof}

%Furthermore, under the additional conditions 
Let us now assume in addition that
\begin{itemize}
    \item $d=3$;
    %\item $F_1,F_2$ are contained in a split Cartan subgroup $A$;
    \item $F_1^+$ is regular (that is, $\Lie(F_1^+)$ is contained in an open Weyl chamber).
\end{itemize}
We are going to show that any such nonzero limit measure $\mu$ has a positive portion of $A$-ergodic components in the class of Haar measure. {(Note that since $d=3$, {the group $A$ is as in  \equ{def-a}, and any such $\mu$ is invariant under the action of $\la F_1,F_2\ra=A$.})}

%\comm{I think it will help the reader if you state explicitly that any such $\mu$ will be $A$-invariant.}\neww{Thanks. Let me do it.}

{In the following we introduce two ergodic decomposition of a finite $A$-invariant measure $\mu$ in an intrinsic way, with the aid of conditional measures. Let $\E_a$ (resp. $\E_A$) denote a countably-generated $\sigma$-algebra equivalent to the $\sigma$-algebra of $a$-invariant (resp. $A$-invariant) Borel subsets in $X_3$. It is known (see \cite[Section 5.14]{EL10}) that the family of conditional measures $\{\mu^{\E_a}_x:x\in X\}$ (resp. $\{\mu^{\E_A}_x:x\in X\}$) is the family of $a$-ergodic (resp. $A$-ergodic) components of $\mu$, where the ergodic decompositions are given by $$
\mu=\int_{X_3}\mu^{\E_a}_x\mathrm{d}\mu(x)=\int_{X_3}\mu^{\E_A}_x\mathrm{d}\mu(x).
$$ 
Note that {by definition} $\mu^{\E_a}_x$ (resp. $\mu^{\E_A}_x$) is always a probability measure. In particular, if {$\widehat{\mu} = \frac1{\mu(X)}\mu$} is the normalized probability measure of $\mu$, then we have $\mu^{\E_a}_x=\widehat{\mu}^{\E_a}_x$ {and} $\mu^{\E_A}_x=\widehat{\mu}^{\E_A}_x$.
}

%\neww{
\begin{lemma}\label{L:entropycon}
    Let $\mu$ be a finite $a$-invariant measure on $X_3$, and let $U$ be a closed subgroup of $H$ normalized by $a$. Then for $\mu$-almost $x\in X_3$ we have $$
    D_{\mu}(a,U)(x)\leq h_{\mu^{\E_a}_x}(a).
    $$
    %where $\E_a$ is the $\sigma$-algebra of $a$-invariant sets.
\end{lemma}
%\comm{Sorry, here you introduced  new notation and lost the reader. If $\mu$ is a measure on $X$ and $x\in X$, how do you  define $\mu_x^{\E_a}$? I think previously you need to introduce ergodic decomposition and  conditional measures as in \cite[Section 5.14]{EL10}. This will help for the next proposition too.}\comm{Thanks!}
\begin{proof}
    See \cite[Theorem 7.6(ii)]{EL10}.
\end{proof}
%}

\begin{proposition}\label{partialhaar}
	Let $\nu$ be an $F_1$-invariant, $a$-ergodic probability measure on $X_3$ with $h_{\nu}(a)>h_{top}(a)-\epsilon(\eta,a)$ as above. {Let $\mu$ be any nonzero weak-* limit of $\mu_{T}$ as $T\to +\infty$. Then there exists a measurable subset $X'\subseteq X_3$ with $\mu(X')>0$, such that for any $x\in X'$, the $A$-ergodic component $\mu^{\E_A}_x$ of $\mu$ equals $m_{X_3}$.
    }

    %Below are original statements:
%	Then any nonzero weak-* limit $\mu$ of 
%	$\mu_T$ as $T\to +\infty$ has a positive proportion of $A$-ergodic components equal to $cm_{X_3}$, where $m_{X_3}$ is the Haar measure on $X_3$ and $0<c\leq 1$. \comm{I am not sure I understand this statement. What exactly does that mean? I think we need to make it more precise. Write the ergodic decomposition and say exactly what is positive. And it is not clear from the statement whether or not $c$ depends on $\mu$.}
\end{proposition}
\begin{proof}
    Without loss of generality, we may assume that %$A$ is the positive diagonal subgroup and that 
    the diagonal components of $a$ are arranged in a strictly descending order. Then $U_{ij}=\exp(\RR E_{ij})$ for $(i,j)\in \{(1,2),(1,3),(2,3)\}$ are three Lyapunov subgroups of $H$. By the choice of $\epsilon=\epsilon(\eta,a)$, we see that $D_{\nu}(a,U)>0$ for the fastest Lyapunov subgroup $U=U_{13}$. 
    Moreover, it is clear that $U$ is normalized by $a$ and $F_2$. Let $\mu$ be a nonzero weak-${}^*$ limit of $\mu_{T_i}$ along a sequence $T_i\to +\infty$. %and write $\mu(X_3)=c\in (0,1]$. The normalized measure of $\mu$ is denoted by $\widehat{\mu}=c^{-1}\mu$. 
   Then it follows from %\cite[Theorem 7.6]{EL10} and Lemma \ref{L:posentropy} %\comm{(Sorry, I think here we again need more detail. A referee should not be forced to look up \cite[Theorem 7.6]{EL10}. Can we state it before the proposition?} \neww{Thanks. Let me do it below.}
    % \neww{The normalized measure of $\mu$ is denoted by $\widehat{\mu}=c^{-1}\mu$. Then it follows from 
    Lemmas \ref{L:posentropy} and \ref{L:entropycon} that for $\mu$-almost $x\in X_3$ we have $$h_{\mu^{\E_a}_x}(a)\geq D_{\mu}(a,U)(x)>0.$$ %where $\E_a$ is the $\sigma$-algebra of $a$-invariant sets \comm{(this notation has already been introduced in the lemma above)}. 
    
    Let $\widehat{\mu}$ denote the normalized probability measure of $\mu$. Then by the convexity of measure-theoretic entropy, we have $$\int_{X_3}h_{\mu^{\E_A}_x}(a)\,\mathrm{d}\widehat{\mu}(x)=h_{\widehat{\mu}}(a)=\int_{X_3}h_{\mu^{\E_a}_x}(a)\,\mathrm{d}\widehat{\mu}(x)>0.$$
   %{Let $\E_A$ denote the $\sigma$-algebra of $A$-invariant sets, and let $\{\widehat{\mu}^{\E_A}_x:x\in X\}$ be the family of $A$-ergodic components of $\widehat{\mu}$ \comm{(this notation should have been introduced before the proposition)}. 
   In particular, the following measurable subset has positive $\widehat{\mu}$-measure: $$
    X':=\{x\in X_3:h_{\mu^{\E_A}_x}(a)>0\}.
    $$
    %Suppose that $\mu(X')=0$. Then it follows from the $A$-ergodic decomposition $\mu=\int_{X_3}\mu^{\E_A}_x\,\mathrm{d}\widehat{\mu}(x)$ that $$h_{\mu}(a)=\int_{X_3}h_{\mu^{\E_A}_x}(a)\,\mathrm{d}\mu(x)=0,$$which is a contradiction. So we must have $\mu(X')>0$. 
    For each $x\in X'$, since $\mu^{\E_A}_x$ is an $A$-invariant, $A$-ergodic probability measure with $h_{\mu^{\E_A}_x}(a)>0$, we see from \cite[Corollary 1.4]{EKL06} that it equals $m_{X_3}$. This completes the proof.
    %Below are your original comments: So we conclude from \cite[Corollary 1.4]{EKL06} \comm{(Sorry, I think here we again need more detail. The corollary says that an $A$-invariant and ergodic measure is Haar; to which measure are we applying it? to some ergodic components? and then integrate the entropy? let us do this explicitly to be safe!)} that $\widehat{\mu}$ has a positive proportion of $A$-ergodic components equal to $m_{X_3}$. This completes the proof.
\end{proof}

In summary, we always choose $\epsilon>0$ small enough such that Propositions \ref{nonzero} and \ref{partialhaar} hold, and in this case we briefly say that $h_{\nu}(a)$ is large enough.

\subsection{Step 3: Marstrand's slicing arguments}

In this subsection we first show that points with dense orbits are generic with respect to a probability measure whose average limit along the orbit has Haar components. The proof below is identical to that of \cite[Theorem 3.2]{bet15}.

\begin{proposition}\label{fulldense}
	Let $\nu$ be a $F_1$-invariant, $a$-ergodic probability measure on $X_3$, and $\mu$ be any weak-* limit of $\mu_T$ as $T\to +\infty$. Suppose that $\mu$ has a positive proportion of $A$-ergodic components equal to $m_{X_3}$. %where $m_{X_3}$ is the Haar measure on $X_3$ and $0<c\leq 1$. 
    Then $\nu\big(D(F_2^+)\big)=1$.
\end{proposition}
\begin{proof}
Suppose to the contrary that $\nu\big(D(F_2^+)\big)<1$. Since $a$ commutes with $F_2^+$, the set $D(F_2^+)$ is $a$-invariant. Then it follows from $a$-ergodicity of $\nu$ that $\nu\big(D(F_2^+)\big)=0$, i.e. $\nu\big(ND(F_2^+)\big)=1$.

Let us consider a decomposition of $\nu|_{ND(F_2^+)}$ in the following manner. First we fix a countable dense subset $\{x_i\}_{i\in\NN}$ of $X_3$. For any two pairs $(i,n),(i',n')$ of natural numbers, we define $(i',n')<(i,n)$ if either $i'+n'<i+n$, or $i'+n'=i+n$ and $i'<i$. Now we define inductively that \begin{equation*}
    \begin{aligned}
        &ND(1,1):=\{x\in X_3:\dist(\overline{F_2^+x},x_1)\geq 1\},\\
        &ND(i,n):=\{x\in X_3:\dist(\overline{F_2^+x},x_i)\geq \frac{1}{n}\}\setminus \bigcup_{(i',n')<(i,n)}ND(i',n')
    \end{aligned}
\end{equation*}
for all $(i,n)\in\NN\times\NN$, where $\dist$ is the Riemannian metric on $X_3$. It is clear that all $\bigcup_{(i',n')<(i,n)}ND(i',n')$ are closed sets, and that $ND(F_2^+)=\bigsqcup_{(i,n)\in\NN\times\NN}ND(i,n)$. Then we may decompose $\nu$ into $\nu=\sum_{(i,n)\in\NN}\nu_{(i,n)}$ for $\nu_{(i,n)}=\nu|_{ND(i,n)}$.

Using {the} Banach--Alaoglu theorem, from the subsequence used to define $\mu$ we may choose a further subsequence $T_k\to +\infty$ such that for all $(i,n)\in\NN\times \NN$, the average \begin{equation*}
    \frac{1}{T_k}\int_0^{T_k}(g^2_t)_*\nu_{(i,n)}\mathrm{d}t
\end{equation*}
converges in the weak-* topology to an $F_2^+$-invariant measure $\mu_{(i,n)}$. By disjointness of the sets $\{ND(i,n):(i,n)\in\NN\times \NN\}$ we obtain that $\mu=\sum_{(i,n)\in\NN}\mu_{(i,n)}$.

    Now we claim that for any $(i,n)\in\NN\times\NN$, the measure $\mu_{(i,n)}$ is singular to the Haar measure $m_{X_3}$. In fact, since $\nu_{i,n}$-almost $x\in X_3$ satisfies $\dist(\overline{F_2^+x},x_i)\geq \frac{1}{n}$, it follows that for any $t>0$ the measure $(g^2_t)_*\nu_{(i,n)}$ gives zero mass to the open ball $B^{X_3}(x_i,\frac{1}{n}):=\{y\in X_3:\dist(y,x_i)<\frac{1}{n}\}$, so does the limit measure $\mu_{(i,n)}$. By $F_2^+$-invariance of $\mu_{(i,n)}$ we see that 
\begin{equation*}
    \bigcup_{k\in\NN} g^2_k\cdot B^{X_3}(x_i,\frac{1}{n})
\end{equation*}
is a $\mu_{(i,n)}$-null set. However, it is also a $m_{X_3}$-full set in view of the ergodicity of $g^2_1$. This verifies the claim.

Finally, it follows from the claim that $\mu=\sum_{(i,n)\in\NN}\mu_{(i,n)}$ is singular to %the Haar measure 
$m_{X_3}$. This contradicts the assumption that $\mu$ has a positive proportion of $A$-ergodic components equal to $m_{X_3}$. The proof is complete.
\end{proof}

Then we apply previous discussions to countably many directions with dense orbits: let $F_{2,j}\,(j\in\NN)$ be one-parameter subgroups contained in the same split Cartan subgroup of $\SL_3(\RR)$ as $F_1$, such that $F_1\neq F_{2,j},\;\forall j\in\NN$. 

\begin{corollary}\label{countfulldense}
    Let $\nu$ be an $F_1$-invariant, $a$-ergodic probability measure on $X_3$ with large enough entropy. Then $\nu\left(\bigcap_{j\in\NN}D(F^+_{2,j})\right)=1$.
\end{corollary}
\begin{proof}
    This follows from Propositions \ref{nonzero}, \ref{partialhaar} and \ref{fulldense}.
\end{proof}

Finally, applying  Marstrand's slicing theorem as in \cite[Section 4]{bet15}  yields the desired conclusion.

\begin{proof}[Proof of Theorem \ref{MainThm1}]
For any $h\in H^0$, since $h$ commutes with the flow $F_1$, the pushforward $(L_h)_*\nu$ of $\nu$ under the left translation by $h$ is an $F_1$-invariant probability measure, which is supported on $$h\cdot B(F_1,K_\epsilon)=B(F_1,h\cdot K_\epsilon)\subseteq B(F_1),$$ and has the same entropy as $\nu$ does. Then it follows from Corollary \ref{countfulldense} that $(L_h)_*\nu(D)=1$ where $D:=\bigcap_{j\in\NN}D(F_{2,j}^+)$. Equivalently, for any $h\in H^0$ and $\nu$-almost $x\in X_3$, we have $hx\in B(F_1)\cap D$. Applying Fubini's theorem gives a $\nu$-full subset $C_1\subseteq X_3$ such that for all $x\in C_1$ and Haar-almost $h\in H^0$, we have $hx\in B(F_1)\cap D$. In particular, for any $x\in C_1$ %and any nonempty open subset $V^0$ of $H^0$, 
it holds that
$$
\dim(B(F_1)\cap D\cap H^0x)=\dim(H^0).
$$

By Lemma \ref{L:localdim}, there exists two $\nu$-full subsets $C_2,C_3\subseteq X_3$ such that \linebreak %for any two nonempty open subsets $V,V^-$ in $H,H^-$ respectively, one has 
$\dim(C_1\cap Hy)\geq d^+(\nu)$ for all $y\in C_2$, and $\dim(C_2\cap H^-z)\geq d^-(\nu)$ for all $z\in C_3$. It follows from Marstrand's slicing theorem that $$
\dim(C_1\cap HH^-z)\geq d^+(\nu)+d^-(\nu)
$$
for all $z\in C_3$, and furthermore that $$
\dim(B(F_1)\cap D\cap H^0HH^-z)\geq \dim(H^0)+d^+(\nu)+d^-(\nu)
$$
for all $z\in C_3$. 
%Now for any fixed nonempty open subset $Y\subseteq X_3$, one may find some $z\in C_3$ and nonempty open subsets $V^0,V,V^-$ in $H^0,H,H^-$ respectively, such that $Y\supseteq V^0VV^-z$. 
By applying Lemma \ref{L:dimclose} to the measure $\nu$, we conclude that 
\begin{equation*}
\begin{aligned}
    \dim(B(F_1)\cap D)&\geq \dim(B(F_1)\cap D\cap H^0HH^-z)\\
    &\geq \dim(H^0)+\dim(H^+)+\dim(H^-)-2\eta\\
    &=\dim(G)-2\eta.\\
    \end{aligned}
\end{equation*}
The arbitrariness of $\eta>0$ completes our proof.
\end{proof}

\section{A nonconstructive argument: Proof of Theorem \ref{Main:ans0}} \label{nonconstr}
In this section, we first present a more general method to find a dense subset of $X_3$ in which each point has a $\lambda$-fiberwise nondivergent $A_C$-orbit and escapes countably many closed subsets with empty interior. Then Theorem \ref{Main:ans0} will be deduced from a combination of this method and Mahler's compactness criterion.

{First we recall some basic notions in Baire categories. In a topological space $X$, a set $E$ is called \textit{meager} (or \textit{of first category}), if it is a countable union of nowhere dense subsets; it is called \textit{non-meager} (or \textit{of second category}) otherwise. We say that a set $E$ is \text{locally non-meager} if its intersection with any non-empty open set is non-meager. This is equivalent to saying that $E$ minus any meager set is dense.}

\begin{proposition}\label{P:nondiv}
    Let $C\subset\mathfrak{a}$ be an open cone\ignore{which is not a half-plane}, and let $\lambda$ be a linear functional compatible with $C$ that is not a root. %Let $\{Z_i\}_{i\in\NN}$ be a sequence of closed subsets with empty interior in $X_3$. Then there exists a dense subset of $X_3$ in which each point has a $\lambda$-fiberwise nondivergent $A_C$-orbit and escapes all $\{Z_i\}_{i\in\NN}$.
    Then the set of points with $\lambda$-fiberwise nondivergent $A_C$-orbits is locally non-meager.
\end{proposition}
\begin{proof}
    We choose and fix a vector $\bv_0\in\ker(\lambda)\setminus \{0\}$ and a vector $\bw_0\in\partial C$ with $\lambda(\bw_0)=1$, such that the cone $C$ is contained in the cone with sides $\RR^+\bv_0$ and $\RR^+\bw_0$. Then the other side of the cone $C$, except $\RR^+\bw_0$, is given by $\RR^+(c_0\bv_0+\bw_0)$ for some $c_0>0$. Using this parametrization, we have $$
L_{C,\lambda,T}=\{\exp(T\bw_0+s\bv_0): s\in (0,c_0T)\}.
$$

	Let us write $a_s=\exp(s\bv_0)$ and $g_t=\exp(t\bw_0)$ for $s,t\in\RR$. Since $\bv_0\in\ker(\lambda)$ is not parallel with any Weyl walls, the ray $\Lie(\{a_s\}_{s>0})=\RR^+\bv_0$ is contained in an open Weyl chamber. Then it follows from Theorem \ref{MainThm1} that the set \begin{equation*}
		\{x\in X_3:
		\{a_sx\}_{s\in\RR} \text{ is bounded and } \{g_tx\}_{t<0} \text{ is dense}
		\}
	\end{equation*}
	is nonempty. We choose and fix any point $x_0$ in the above set, and denote by $K$ a compact neighborhood of $K_0:=\overline{\{a_sx_0\}_{s\in\RR}}$. Note that 
	\begin{equation*}
		L_{C,\lambda,T}=\{a_sg_T:s\in (0,c_0T)\}.
	\end{equation*}
	From this we shall construct a locally non-meager subset of points with $\lambda$-fiberwise nondivergent $A_C$-orbits.
	
	%It is well-known that there are only countably many compact $A$-orbits in $X_3$, say $Az_1,Az_2,\cdots$. 
    We choose and fix any nonempty open subset $\cN_1$ in $X_3$. Let $Z\subseteq X_3$ be a meager set, namely, $Z=\bigcup_{i\in\NN}Z_i$ where each $Z_i$ is nowhere dense in $X_3$. Then by density there exists some $T_1>0$ with $y_1:=g_{-T_1}x_0\in \cN_1\setminus Z_1$. It follows that the orbit \begin{equation*}
		L_{C,\lambda,T_1}y_1=\{a_sx_0:s\in (0,c_0T_1)\}
	\end{equation*}
	is contained in $K_0$. By continuity we may find an open neighborhood $\cN_2$ of $y_1$ (depending on $T_1$), such that $\overline{\cN_2}\subseteq \cN_1\setminus Z_1$ and that \begin{equation*}
		z\in \cN_2\Longrightarrow L_{C,\lambda,T_1}z\subseteq K.
	\end{equation*}
	Then by density there exists some $T_2>2T_1$ with $y_2:=g_{-T_2}x_0\in \cN_2\setminus Z_2$.
	It follows that the orbit \begin{equation*}
		L_{C,\lambda,T_2}y_2=\{a_sx_0:s\in (0,c_0T_2)\}
	\end{equation*}
	is contained in $K_0$. By continuity we may find an open neighborhood $\cN_3$ of $y_2$ (depending on $T_2$), such that $\overline{\cN_3}\subseteq \cN_2\setminus Z_2$ and that \begin{equation*}
		z\in \cN_3\Longrightarrow L_{C,\lambda,T_2}z\subseteq K.
	\end{equation*}
	
	Continuing this process gives a point $z_*$ in $\bigcap_{i\in \NN}\overline{\cN_i}=\bigcap_{i\in \NN}\cN_i$ satisfying the following properties:
	\begin{itemize}
		\item for any $i\in\NN$, the orbit $L_{C,\lambda,T_i}z_*$ lies in $K$.
		\item for any $i\in\NN$, the point $z_*$ doesn't lie in $Z_i$.
	\end{itemize}
	In particular this means that $z_*\in \cN_1$ has a $\lambda$-fiberwise nondivergent $A_C$-orbit and avoids $Z=\bigcup_{i\in\NN}Z_i$. By arbitrariness of $\cN_1$ the proof is complete.
\end{proof}

Next, we shall show that points with bounded $F^+$-orbits for some ray in $A$ are contained in a meager set of $X_3$.

\begin{lemma}\label{L:cpthaus}
    Let $K\subseteq X_3$ be a compact subset, and write $$
    Z:=\{x\in X_3:F^+x\subseteq K\text{ for some ray }F^+\subseteq A\}.
    $$
    Then $Z$ is a closed subset with empty interior in $X_3$.
\end{lemma}
\begin{proof}
We first show that $Z$ is a closed set. Suppose that $\{x_n\}_{n\in\ZZ}\subseteq Z$ has a limit point $x\in X_3$. Then there exists a sequence of unit vectors $\{\bv_n\}_{n\in\NN}\subseteq\mathfrak{a}$ such that $\exp(\RR^+\bv_n)x_n\subseteq K$ for each $n\in\NN$. By passing to a subsequence, we may assume that $\bv_n\to \bv$ as $n\to +\infty$ and $\bv$ is a unit vector in $\mathfrak{a}$. It follows that for any $t>0$, $$
K\ni \exp(t\bv_n)x_n\to \exp(t\bv)x\text{ as }n\to +\infty,
$$
which means that $F^+x\subseteq K$ for the ray $F^+:=\exp(\RR^+\bv)\leq A$. This verifies that $x\in Z$.

Next we show that $Z$ has empty interior. Let us decompose $\mathfrak{a}\setminus\{0\}$ into a disjoint union of open Weyl chambers and their walls. Since a finite union of closed subsets with empty interior still has empty interior, it suffices to show that for any open Weyl chamber $D_0\subseteq\mathfrak{a}$ and any unit vector $\bv\in\partial D_0$, the closed subsets \begin{equation*}
\begin{aligned}
    Z'&:=\{x\in X_3:F^+x\subseteq K\text{ for some ray }F^+\subseteq \exp(D_0)\},\\
    Z''&:=\{x\in X_3:F^+x\subseteq K\text{ for the ray }F^+=\exp(\RR^+\bv)\}
\end{aligned}
\end{equation*}
both have empty interior. It is clear that the interior of $Z''$ is empty, for $Z''$ has zero Haar measure by the ergodicity of $F^+$.

Now we claim that $X_3\setminus Z'$ contains all rational points $\{g\Gamma\in X_3:g\in\SL_3(\QQ)\}$. This is a dense subset in $X_3$, which implies that the interior of $Z'$ is empty as desired. In fact, by a rational conjugation, we may assume that $$
D_0=\{\diag(t_1,t_2,-t_1-t_2):t_1>t_2>-t_1-t_2\}.
$$
Let $B\leq G$ be the subgroup of lower triangular matrices, and $U\leq G$ be the subgroup of unipotent upper triangular matrices. Since $BU$ is Zariski open in $G$ and $\Gamma$ is Zariski dense in $G$, we have $G=BU\cdot \Gamma$, and hence that $$\{g\Gamma\in X_3:g\in\SL_3(\QQ)\}
=\{bu\Gamma\in X_3:b\in B(\QQ),u\in U(\QQ)\}.$$
For any ray $F^+=\{g_t\}_{t>0}\subseteq \exp(D_0)$ where $g_t=\exp(e^{\alpha t},e^{(1-\alpha)t},e^{-t})$ with $\alpha\in (\frac{1}{2},1)$, we see that $$
F^+(bu\Gamma)\text{ is unbounded}\Longleftrightarrow F^+(u\Gamma)\text{ is unbounded}.
$$
Moreover, for any $u\in U(\QQ)$, we may choose some $z=(p_1,p_2,q)^t\in\ZZ^2\times \NN$ such that $u\cdot z=(0,0,q)^t$, which implies that $$
\min_{z\in \ZZ^3\setminus \{0\}}\|g_tuz\|\leq e^{-t}q\to 0\text{ as }t\to +\infty.
$$
In view of Mahler's compactness criterion, this means that $F^+(u\Gamma)$ is unbounded. We conclude that for any rational point $x\in X_3(\QQ)$ and any ray $F^+\subseteq \exp(D_0)$, the orbit $F^+x$ is unbounded in $X_3$. This verifies the claim and hence completes the proof.
\end{proof}

\begin{proof}[Proof of Theorem \ref{Main:ans0}] Let $\{K_i\}_{i\in\NN}$ be a compact exhaustion of $X_3$, and for each $i\in\NN$ write $Z_i=\{x\in X_3:F^+x\subseteq K_i\text{ for some ray }F^+\subseteq A\}$. It is clear that $$
\{x\in X_3:F^+x\text{ is bounded for some ray }F^+\subseteq A\}=\bigcup_{i\in\NN}Z_i.
$$
It follows from Lemma \ref{L:cpthaus} that the above set is meager.
Then the conclusion follows from Proposition \ref{P:nondiv}. 
\end{proof}

\appendix
\section{Proof of Lemma \ref{L:entropydim}}\label{A:a}
In this section we verify Lemma \ref{L:entropydim} via a direct calculation of entropy. To begin with we define a {metric on the Lie algebra $\mathfrak{g}=\Lie(G)$ specific} to our needs. As in the beginning of Section \ref{horosph}, the adjoint action of any $\Ad$-non-quasi-unipotent element $a$ decomposes $\mathfrak{g}$ into a direct sum of real generalized eigenspaces: $$
\Lg=\bigoplus_{\lambda\in{\sigma_a}\cap\RR}(E_{\lambda}\cap\Lg)\oplus \bigoplus_{\lambda\in{\sigma_a}\setminus\RR}((E_{\lambda}+E_{\overline{\lambda}})\cap\Lg).
$$
Since there is no need to distinguish between generalized eigenspaces for real and non-real eigenvalues, we index them as $$
\Lg=\bigoplus_{i=1}^mE_i,
$$
where the indices $i=1,\cdots,k$ correspond to the expanding generalized eigenspaces, $i=k+1,\cdots,m-k$ correspond to the central generalized eigenspaces, and $i=m-k+1,\cdots,m$ correspond to the contracting generalized eigenspaces, such that the corresponding eigenvalues for each generalized eigenspaces are ordered: $$
|\lambda_1|\geq\cdots\geq|\lambda_k|>1=|\lambda_{k+1}|=\cdots=|\lambda_{m-k}|>|\lambda_{m-k+1}|\geq\cdots\geq|\lambda_m|.
$$
For each generalized eigenspace $E_i$, fix an orthonormal basis and impose the sup norm $\|\cdot\|_i$. These norms induce a metric on $\Lg$ via
$$d^\Lg(v,v'):=\max_{1\leq i\leq m}\|v_i-v_i'\|_i
$$
for vectors $v=\sum_{i=1}^mv_i$ and $v'=\sum_{i=1}^mv_i'$ (where $v_i,v_i'\in E_i$).

The above metric $d^{\mathfrak{g}}$ on $\Lg$ naturally defines a right-invariant metric $d^G$ on $G$, and also induces a metric $d^X$ on $X=G/\Gamma$. We write \begin{equation*}
    \begin{aligned}
        &B^{\Lg}(v,r):=\{v'\in \Lg:d^\Lg(v,v')<r\};\\
        &B^{G}(g,r):=\{g'\in G:d^G(g,g')<r\};\\
        &B^{X}(x,r):=\{x'\in X:d^X(x,x')<r\}
    \end{aligned}
\end{equation*}
for any $r>0$ and $v\in\Lg,g\in G,x\in X$.

Let $K\subseteq X$ be a compact $a$-invariant set, and take $\epsilon_0>0$ to be any number smaller than the injectivity radius of $K$. This means that for any $x\in K$, the map \begin{equation}\label{E:iso0}
B^G(1_G,\epsilon_0)\to B^X(x,\epsilon_0),\quad g\mapsto gx
\end{equation}
is an isometry. To compute $h_{top}(a|_K)$ we consider the $n$-Bowen balls in $X$ centered in $K$ with respect to the metric $d^X$: $$D^X_n(x,\epsilon):=\{y\in X:\forall 0\leq j\leq n-1,\;d^X(a^jy,a^jx)\leq \epsilon\},$$ and correspondingly, the $n$-Bowen balls in $G$ centered at $1_G$ with respect to the metric $d^G$:
$$D^G_n(1_G,\epsilon):=\{g\in G:\forall 0\leq j\leq n-1,\;d^G(a^jga^{-j},1_G)\leq \epsilon\}.$$
The first lemma says when restricted on $K$ one may replace the $n$-Bowen balls in $X$ by the image of the $n$-Bowen balls in $G$ under the map (\ref{E:iso0}).

\begin{lemma}
    For any sufficiently small $\epsilon=\epsilon(\epsilon_0,a)>0$, any $x\in K$ and any $n\geq 1$, the map \begin{equation}\label{E:iso1}
    D^G_n(1_G,\epsilon)\to D^X_n(x,\epsilon),\quad g\mapsto gx\end{equation} is an isometric surjection.
\end{lemma}
\begin{proof}
    Since the conjugate map $g\mapsto aga^{-1}$ is continuous on $B^G(1_G,\epsilon_0)$, we may choose $\epsilon=\epsilon(\epsilon_0,a)<\epsilon_0$ small enough such that $a\cdot D_1^G(1_G,\epsilon)\cdot a^{-1}\subseteq B^G(1_G,\epsilon_0)$. Then it follows from $D^G_n(1_G,\epsilon)\subseteq B^G(1_G,\epsilon_0)$ that the map (\ref{E:iso1}) is an isometry. Moreover, we see that for any $g\in D^G_n(1_G,\epsilon)$ and any $0\leq j\leq n-1$, $$
    d^X(a^jgx,a^jx)=d^X(a^jga^{-j}\cdot a^jx,a^jx)\leq d^G(a^jga^{-j},1_G)\leq \epsilon,
    $$ which means that $gx\in D^X_n(x,\epsilon)$. It suffices to check that the map (\ref{E:iso1}) is surjective, namely, any $y\in D^X_n(x,\epsilon)$ has the form $gx$ for some $g\in D^G_n(1_G,\epsilon)$.

    In fact, since $y\in D^X_n(x,\epsilon)\subseteq B^X(x,\epsilon_0)=B^G(1_G,\epsilon_0).x$, one may write $y=gx$ for some $g\in B^G(1_G,\epsilon_0)$. We prove by induction on $n$ that $g\in D^G_n(1_G,\epsilon)$. First, when $n=1$, it follows from $gx\in D_1^X(x,\epsilon)$ and the map (\ref{E:iso0}) (applying to $x\in K$) that $g\in D^G_1(x,\epsilon)$. Suppose that $n\geq 2$ and $g\in D^G_{n-1}(1_G,\epsilon)$. In particular, it follows from $a^{n-2}ga^{-(n-2)}\in D^G_1(1_G,\epsilon)$ that $a^{n-1}ga^{-(n-1)}\in B^G(1_G,\epsilon_0)$. Since $$
    d^X((a^{n-1}ga^{-(n-1)})a^{n-1}x,a^{n-1}x)=d^X(a^{n-1}y,a^{n-1}x)\leq\epsilon,
    $$
    we see from the map (\ref{E:iso0}) (applying to $a^{n-1}x\in K$) that $a^{n-1}ga^{-(n-1)}\in D^G_1(1,\epsilon)$, and hence conclude that $g\in D^G_n(1_G,\epsilon)$. This completes the proof.
\end{proof}

We also consider the $n$-Bowen balls in $\Lg$ centered at $0_\Lg$ with respect to the metric $d^\Lg$:
$$D^\Lg_n(0_\Lg,\epsilon):=\{v\in \Lg:\forall 0\leq j\leq n-1,\;d^\Lg(\Ad(a)^jv,0_\Lg)\leq \epsilon\}.$$ The next lemma says one may furthermore replace the $n$-Bowen balls in $G$ centered at $1_G$ by the image of the $n$-Bowen balls in $\mathfrak{g}$ centered at $0_\Lg$ under the exponential map.

\begin{lemma}
    There exists a constant $C_0=C_0(G)>0$ such that for any sufficiently small $\epsilon>0$ and any $n\geq 1$, the maps
    \begin{equation*}
        \begin{aligned}
            &D_n^\Lg(0_\Lg,C_0^{-1}\epsilon)\to D_n^G(1_G,\epsilon),\quad v\mapsto \exp(v)\\
            &D_n^G(1_G,\epsilon)\to D_n^\Lg(0_\Lg,C_0^{-1}\epsilon),\quad g\mapsto \log(g)
        \end{aligned}
    \end{equation*}
    are well-defined smooth bi-Lipschitz diffeomorphisms onto an open subset.
\end{lemma}
\begin{proof}
    This easily follows from the fact that the exponential map is a local smooth bi-Lipschitz diffeomorphism around $0_\Lg$.
\end{proof}

To summarize, any sufficiently small $n$-Bowen ball in $X$ centered in $K$ is ``roughly'' the image of a small $n$-Bowen ball in $\Lg$ with radius of the same order, i.e. there exists a constant $C_0=C_0(G)>0$ such that for any sufficiently small $\epsilon>0$, any $x\in K$ and $n\geq 1$, $$
\exp(D^\Lg_n(0_\Lg,C_0^{-1}\epsilon)).x\subseteq D^X_n(x,\epsilon)\subseteq\exp(D^\Lg_n(0_\Lg,C_0\epsilon)).x.
$$

Here a lemma is inserted to describe the actual shape of open balls $B^\Lg(v,r)$ and $n$-Bowen balls $D^\Lg_n(0_\Lg,\eps)$ with respect to the metric $d^\Lg$.

\begin{lemma}\label{L:parallel}
    With respect to the fixed basis and the metric $d^\Lg$ in $\Lg$, each open ball $B^\Lg(v,r)$ is an open cube centered at $v$, of side length $2r$, and with edges parallel to the basis vectors. Each $n$-Bowen ball $D^\Lg_n(0_{\Lg},\epsilon)$ is contained in a closed parallelepiped centered at $0_{\Lg}$ whose edges parallel to the basis vectors are all equal in length to \begin{itemize}
        \item $2C_1\epsilon\cdot(|\lambda_i|-\delta)^{-(n-1)}$ for $1\leq i\leq k$;
        \item $2C_1\epsilon$ for $k+1\leq i\leq m$,
    \end{itemize}
    where $C_1=C_1(\delta,a)>0$ is a constant given by any $0<\delta<|\lambda_k|-1$.
\end{lemma}
\begin{proof}
    The first statement is clear from definition. For the second statement, we first observe from Jordan's canonical form that for any $\delta>0$ there exists a constant $C_1=C_1(\delta,a)>0$ such that $$
    C_1^{-1}\cdot (|\lambda_i|-\delta)^n\|v\|_i\leq \|\Ad(a)^nv\|_i\leq C_1\cdot (|\lambda_i|+\delta)^n\|v\|_i
    $$
    for any $i\geq 1,n\geq 0$ and any $v\in E_i$. Choose $0<\delta<|\lambda_k|-1$. It follows that \begin{equation*}
    \begin{aligned}
        D^\Lg_n(0_{\Lg},\eps)&=\left\{v=\sum_{i=1}^mv_i\in\Lg: \forall 1\leq i\leq m,\forall 0\leq j\leq n-1,\|\Ad^j(a)v_i\|_i\leq \epsilon\right\}\\
        &\subseteq\left\{v=\sum_{i=1}^mv_i\in\Lg: \forall 1\leq i\leq m,\|v_i\|_i\leq C_1\epsilon\cdot\min_{0\leq j\leq n-1}(|\lambda_i|-\delta)^{-j}\right\}\\
        &=\left\{v=\sum_{i=1}^mv_i\in\Lg: 
        \begin{matrix}
\forall 1\leq i\leq k,\|v_i\|_i\leq C_1\epsilon\cdot(|\lambda_i|-\delta)^{-(n-1)};\\
\forall k+1\leq i\leq m,\|v_i\|_i\leq C_1\epsilon
        \end{matrix}
  \right\}
    \end{aligned}
    \end{equation*}
    This is the desired parallelepiped.
\end{proof}

Now we introduce two notions concerning coverings of $K$, one of which using $n$-Bowen balls is related to the topological entropy of $a$ restricted to $K$, and the other using open balls is related to the Hausdorff dimension of $K$. We shall show that they are in close relation with each other. Write \begin{equation*}
    \begin{aligned}
        &M_K(\epsilon,n):=\text{minimal number of Bowen balls }D_n^X(x,\epsilon) \text{ centered in } K \text{ to cover }K;\\
        &N_K(r):= \text{minimal number of open balls }B^X(x,r) \text{ centered in } X \text{ to cover }K.
    \end{aligned}
\end{equation*}
It follows from definitions that \begin{equation*}
    \begin{aligned}
        &h_{top}(a|_K)\geq\lim_{\eps\to 0^+}\limsup_{n\to +\infty}\frac{\log(M_K(\epsilon,n))}{n};\\
        &\dim(K)\leq \liminf_{r\to 0^+}\frac{\log(N_K(r))}{\log(r^{-1})}.
    \end{aligned}
\end{equation*}

\begin{lemma}
    For any sufficiently small $\epsilon,\delta>0$ and any $n\geq 1$, we have $$C(\epsilon,n)\cdot M_K(\epsilon,n)\geq N_K(r(\epsilon,n)),$$ where 
    \begin{equation*}
        \begin{aligned}
&C(\epsilon,n)\leq 2^d(|\lambda_1|-\delta)^{d(n-1)}\cdot\prod_{i=1}^k\dfrac{1}{(|\lambda_i|-\delta)^{d_i(n-1)}},\\ &r(\epsilon,n)=C_1C_0^2\epsilon\cdot (|\lambda_1|-\delta)^{-(n-1)},\;d=\dim(\Lg),\;d_i=\dim(E_i).
    \end{aligned}
    \end{equation*}
\end{lemma}
\begin{proof}
Let $\epsilon_1>0$ be any fixed number such that
\begin{itemize}
    \item $\exp|_{B^\Lg(0_\Lg,\epsilon_1)}$ is a smooth bi-Lipschitz diffeomorphism onto its image.
    \item $\exp(B^\Lg(0_\Lg,\epsilon_1))$ is contained in the ball $B^G(1_G,\epsilon_0)$ given by (\ref{E:iso0}).
\end{itemize} Let $C_0=C_0(G)>0$ denote the bi-Lipschitz constant of the exponential map around $0_\Lg$, i.e. for any $v,v'\in B^\Lg(0_\Lg,\epsilon_1)$, $$
    C_0^{-1}\cdot d^G(\exp(v),\exp(v'))\leq d^\Lg(v,v')\leq C_0\cdot d^G(\exp(v),\exp(v')).
    $$
    It follows that for any $x\in K$, the composition of maps $$
    \phi_x: B^\Lg(0_\Lg,\epsilon_0)\to \exp(B^\Lg(0_\Lg,\epsilon_0)).x,\quad v\mapsto\exp(v)x
    $$
    is a smooth bi-Lipschitz diffeomorphism with constant $C_0$. Suppose that $\epsilon>0$ is small enough such that $C_0\cdot(\epsilon+r(\epsilon,n))<\epsilon_1$. To prove the lemma, it suffices to show that each Bowen ball $D^X_n(x,\epsilon)$ centered in $K$ can be covered by $\leq C(\epsilon,n)$ open balls $B^X(x',r(\epsilon,n))$ centered in $X$. Note that $D^X_n(x,\epsilon)\subseteq \phi_x(D^\Lg_n(0_\Lg,C_0\epsilon))$ and for $x'=\phi_x(v')\in \phi_x(B^\Lg(0_\Lg,C_0\epsilon))$, $$
    \phi_x(B^\Lg(v',C_0^{-1}\cdot r(\epsilon,n)))\subseteq B^X(x',r(\epsilon,n))\subseteq \phi_x(B^\Lg(v',C_0\cdot r(\epsilon,n))).
    $$
    So we only need to show that each Bowen ball $D^\Lg_n(0_\Lg,C_0\epsilon)$ can be covered by $\leq C(\epsilon,n)$ open balls $B^\Lg(v',C_0^{-1}\cdot r(\epsilon,n))$ centered in $B^\Lg(0_\Lg,C_0\epsilon)$.

    In fact, in view of Lemma \ref{L:parallel}, we may choose \begin{equation*}
        C(\epsilon,n)=\prod_{i=1}^k\left\lceil\dfrac{2C_1C_0\epsilon\cdot (|\lambda_i|-\delta)^{-(n-1)}}{2C_0^{-1}\cdot r(\epsilon,n)}\right\rceil^{d_i}\cdot \prod_{i=k+1}^m\left\lceil\dfrac{2C_1C_0\epsilon}{2C_0^{-1}\cdot r(\epsilon,n)}\right\rceil^{d_i}
    \end{equation*}
    and $r(\epsilon,n)=C_1C_0^2\epsilon\cdot (|\lambda_1|-\delta)^{-(n-1)}$, where the constant $C_1=C_1(\delta,a)$ is given by any $0<\delta<|\lambda_k|-1$. It follows that $$
    C(\epsilon,n)\leq 2^d(|\lambda_1|-\delta)^{d(n-1)}\cdot\prod_{i=1}^k\dfrac{1}{(|\lambda_i|-\delta)^{d_i(n-1)}}.
    $$
\end{proof}

Therefore, we conclude that \begin{equation*}
    \begin{aligned}
        h_{top}(a|_K)&\geq \lim_{\eps\to 0^+}\limsup_{n\to +\infty}\frac{\log(N_K(\epsilon,n))-\log(C(\epsilon,n))}{\log(r(\epsilon,n)^{-1})}\cdot \frac{\log(r(\epsilon,n)^{-1})}{n}\\
        &\geq (\dim(K)-d)\cdot\log(|\lambda_1|-\delta)+\sum_{i=1}^kd_i\log(|\lambda_i|-\delta).
    \end{aligned}
\end{equation*}
Letting $\delta\to 0^+$ gives $h_{top}(a|_K)\geq (\dim(K)-d)\cdot\log|\lambda_1|+h_{top}(a)$.

%He is also grateful to for suggesting a possible approach to Theorem \ref{T:Mstarnonempty}. \comm{Just a  remark:  I am also grateful to Prasuna and Reynold, but I am not sure it was them who suggested that approach. Perhaps it is closer to some ideas suggested by Elon after my talk in September. Anyway I think it is appropriate to thank Elon, and also Barak for some discussions.}

\end{document}